\documentclass[10pt]{amsart}

\usepackage{amssymb}
\usepackage{esint}
\usepackage{geometry, verbatim, graphicx, color}

\usepackage[colorlinks=true,urlcolor=blue, citecolor=red,linkcolor=blue,
linktocpage,pdfpagelabels, bookmarksnumbered,bookmarksopen]{hyperref}

\usepackage[hyperpageref]{backref}

\newtheorem{lm}{Lemma}[section]
\newtheorem{teo}[lm]{Theorem}
\newtheorem{prop}[lm]{Proposition}
\newtheorem{coro}[lm]{Corollary}
\newtheorem*{teoa}{Theorem A}
\newtheorem*{teob}{Theorem B}
\theoremstyle{definition}
\newtheorem{defi}[lm]{Definition}
\newtheorem{oss}[lm]{Remark}
\newtheorem{exa}[lm]{Example}
\newtheorem*{ack}{Acknowledgments}

\numberwithin{equation}{section}

\newcommand\co{\mbox{%
\raisebox{.8ex}{\rm c}%
\kern-.175em\raisebox{.2ex}{/}%
\kern-.18em\raisebox{-.2ex}{\rm o}%
}~}

\author[Brasco]{Lorenzo Brasco}
\address[L.\ Brasco]{Dipartimento di Matematica e Informatica
\newline\indent
Universit\`a degli Studi di Ferrara
\newline\indent
Via Machiavelli 35, 44121 Ferrara, Italy}
\email{lorenzo.brasco@unife.it}

\author[De Philippis]{Guido De Philippis}
\address[G.~De Philippis]{Courant Institute of Mathematical Sciences
\newline\indent
New York University
\newline\indent
251 Mercer St., New York, NY 10012, USA.}
\email{guido@cims.nyu.edu}

\author[Franzina]{Giovanni Franzina}
\address[G.\ Franzina]{Istituto Nazionale di Alta Matematica (In{\rm d}am)
\newline\indent
Unit\`a di Ricerca di Firenze \co DiMaI ``Ulisse Dini'' 
\newline\indent 
Universit\`a degli Studi di Firenze
\newline\indent 
Viale Morgagni 67/A, 50134 Firenze, Italy}
\email{giovanni.franzina@unifi.it}


\title[sublinear Lane-Emden equation]{Positive solutions to the\\sublinear Lane-Emden equation\\ are isolated}

\begin{document}

\begin{abstract}
We prove that on a smooth bounded set, the positive least energy solution of the Lane-Emden equation with sublinear power is isolated. As a corollary, we obtain that the first $q-$eigenvalue of the Dirichlet-Laplacian is not an accumulation point of the $q-$spectrum, on a smooth bounded set. Our results extend to a suitable class of Lipschitz domains, as well.
\end{abstract}

\maketitle

\begin{center}
\begin{minipage}{11cm}
\small
\tableofcontents
\end{minipage}
\end{center}

\section{Introduction}

\subsection{Overview}
We consider an open bounded set $\Omega\subset\mathbb{R}^N$, with its associated {\it homogeneous Sobolev space} $\mathcal{D}^{1,2}_0(\Omega)$. The latter is defined as the completion of $C^\infty_0(\Omega)$ with respect to the norm
\[
\|\varphi\|_{\mathcal{D}^{1,2}_0(\Omega)}=\left(\int_\Omega |\nabla \varphi|^2\,dx\right)^\frac{1}{2},\qquad \text{ for } \varphi\in C^\infty_0(\Omega).
\]
The notation $C^\infty_0(\Omega)$ stands for the set of $C^\infty$ functions with compact support in $\Omega$.
We recall that an open bounded set $\Omega$ supports a Poincar\'e inequality of the type
\begin{equation}
\label{poinpoin}
\frac{1}{C}\,\int_\Omega |\varphi|^2\le \int_\Omega |\nabla \varphi|^2\,dx,\qquad \text{ for every }\varphi\in C^\infty_0(\Omega),
\end{equation}
thus the space $\mathcal{D}^{1,2}_0(\Omega)$ coincides with the closure of $C^\infty_0(\Omega)$ in the standard Sobolev space $W^{1,2}(\Omega)$.
\par
It is well-known that in this setting the Dirichlet-Laplacian operator on $\Omega$ has a discrete spectrum, made of positive eigenvalues accumulating at $+\infty$. In other words, the boundary value problem
\[
-\Delta u=\lambda\,u,\quad \text{ in }\Omega,\qquad u=0, \quad \text{ on }\partial\Omega,
\]
admits non-trivial solutions $u\in\mathcal{D}^{1,2}_0(\Omega)$ only for a discrete set of characteristic values $\lambda$, that we indicate with $0<\lambda_1(\Omega)\le \lambda_2(\Omega)\le \dots$. If $u\in\mathcal{D}^{1,2}_0(\Omega)$ solves the above equation with $\lambda=\lambda_i(\Omega)$, it is called an {\it eigenfunction associated to $\lambda_i(\Omega)$}.
\par
It is easy to see that these eigenvalues $\lambda_i(\Omega)$ can be understood as the critical values of the Dirichlet integral
\[
\varphi\mapsto \int_\Omega |\nabla \varphi|^2\,dx, 
\]
constrained to the manifold
\[
\mathcal{S}_2(\Omega)=\Big\{\varphi\in\mathcal{D}^{1,2}_0(\Omega)\, :\, \|\varphi\|_{L^2(\Omega)}=1\Big\}.
\]
The associated critical points correspond to the eigenfunctions of the Dirichlet-Laplacian, normalized in order to have unit $L^2$ norm. 
In particular, the first eigenvalue corresponds to the global constrained minimum, i.e.
\[
\lambda_1(\Omega)=\min_{\varphi\in\mathcal{D}^{1,2}_0(\Omega)}\left\{\int_\Omega |\nabla \varphi|^2\,dx\, :\, \|\varphi\|_{L^2(\Omega)}=1\right\},
\] 
which in turn gives the sharp constant in \eqref{poinpoin}.
\vskip.2cm\noindent
One may wonder what happens when the constraint $\mathcal{S}_2(\Omega)$ is replaced by the more general one
\[
\mathcal{S}_q(\Omega)=\Big\{u\in\mathcal{D}^{1,2}_0(\Omega)\, :\, \|u\|_{L^q(\Omega)}=1\Big\},
\]
where $q\not =2$ and\footnote{We use the usual notation $2^*$ for the critical Sobolev exponent, i.e.
\[
2^*=\frac{2\,N}{N-2},\ \text{ for } N\ge3,\qquad 2^*=+\infty, \ \text { for }N=2.
\]} $1<q<2^*$.
In this case, by the Lagrange's multipliers rule, the relevant elliptic equation is given by
\[
-\Delta u=\lambda\,|u|^{q-2}\,u,\ \text{ in } \Omega,\qquad u=0, \ \text{ on }\partial\Omega.
\]
If we want to get rid of the normalization on the $L^q$ norm, then the equation should be written in the following form
\begin{equation}
\label{autovalorenonlocal}
-\Delta u=\lambda\,\|u\|_{L^q(\Omega)}^{2-q}\, |u|^{q-2}\,u,\ \text{ in } \Omega,\qquad u=0, \ \text{ on }\partial\Omega.
\end{equation}
We define the {\it $q-$spectrum of the Dirichlet-Laplacian on $\Omega$} as
\[
\mathrm{Spec}(\Omega;q)=\Big\{\lambda\in \mathbb{R}\, :\, \text{equation \eqref{autovalorenonlocal} admits a solution in } \mathcal{D}^{1,2}_0(\Omega)\setminus\{0\}\Big\}.
\]
Each element $\lambda$ of this set is called a {\it $q-$eigenvalue}, while an associated solution of \eqref{autovalorenonlocal} will be called {\it $q-$eigenfunction}.
\par
Some basic properties of this eigenvalue--type problem has been recently collected in the survey paper \cite{BF}. Let us briefly recall them, by referring to \cite{BF} for all the missing details.
\par
First of all, by using standard variational techniques from Critical Point Theory, it is relatively easy to produce an infinite sequence of $q-$eigenvalues, diverging at $+\infty$. For every $k\in\mathbb{N}\setminus\{0\}$, these are given by
\begin{equation}
\label{CFWLS}
\lambda_{k,LS}(\Omega;q)=\inf_{\mathcal{F}\in \Sigma_k(\Omega;q)}\left\{\max_{\varphi\in\mathcal{F}} \int_\Omega |\nabla\varphi|^2\,dx\right\},
\end{equation}
where 
\[
\Sigma_k(\Omega;q)=\Big\{\mathcal{F}\subset \mathcal{S}_q(\Omega)\, :\, \mathcal{F} \text{ compact and symmetric with } \gamma(\mathcal{F})\ge k\Big\},
\]
and $\gamma$ is the {\it Krasnosel'ski\u{\i} genus}, defined by
\[
\gamma(\mathcal{F})=\inf\Big\{k\in\mathbb{N}\setminus\{0\}\, :\, \exists \text{ a continuous odd map } \phi:\mathcal{F}\to\mathbb{S}^{k-1}\Big\}.
\] 
We recall that formula \eqref{CFWLS} is reminiscent of the celebrated {\it Courant-Fischer-Weyl min-max principle} for the eigenvalues of the Laplacian (see \cite[equation (1.32)]{He}).
\par
For $k=1$, it is not difficult to see that formula \eqref{CFWLS} reduces to the sharp Poincar\'e-Sobolev constant
\[
\lambda_1(\Omega;q)=\min_{\varphi\in\mathcal{D}^{1,2}_0(\Omega)} \left\{\int_\Omega |\nabla \varphi|^2\,dx\, :\, \int_\Omega |\varphi|^q\,dx=1\right\},
\]
and that this is the first $q-$eigenvalue of $\Omega$, i.e. 
\[
\lambda_1(\Omega;q)\in \mathrm{Spec}(\Omega;q)\qquad \text{ and }\qquad \lambda\ge \lambda_1(\Omega;q) \text{ for every }\lambda\in\mathrm{Spec}(\Omega;q).
\]
Let us now specialize the discussion to the case $1<q<2$. In this case, for a generic open set we have
\[
\Big\{\lambda_{k,LS}(\Omega;q)\Big\}_{k\in\mathbb{N}\setminus\{0\}}\not=\mathrm{Spec}(\Omega;q),
\]
and $\mathrm{Spec}(\Omega;q)$ is not discrete. Even worse, one can produce examples of sets $\Omega$ for which the first $q-$eigenvalue is not isolated, i.e. it is an accumulation point for $\mathrm{Spec}(\Omega;q)$ (see \cite[Theorem 3.2]{BFcontro}).
\par
Examples of this last phenomenon are quite pathological, i.e. they are sets made of countably many connected components. It is thus reasonable to ask whether $\lambda_1(\Omega;q)$ is isolated or not, for connected sets or sets with a finite number of connected components. 
\par
This leads us to the question tackled in this paper: find classes of ``good'' sets such that the first $q-$eigenvalue is isolated, for $1<q<2$.

\subsection{The Lane-Emden equation}
The equation \eqref{autovalorenonlocal} may look weird at a first sight, but actually it is just a scaled version of the celebrated and well-studied {\it Lane-Emden equation}. 
More precisely, observe that equation \eqref{autovalorenonlocal} is no more linear, but it is still $1-$homogeneous. This means that if $u\in\mathcal{D}^{1,2}_0(\Omega)$ is a solution, then $t\,u$ is still a solution for every $t\in\mathbb{R}$. Thus, if we take a $q-$eigenfunction $u$ such that
\[
\|u\|_{L^q(\Omega)}=\lambda^\frac{1}{q-2},
\]
we get that this solves the usual Lane-Emden equation
\begin{equation}
\label{LE}
-\Delta u=|u|^{q-2}\,u,\ \text{ in }\Omega.
\end{equation}
This semilinear elliptic equation naturally arises in many fields, here we just want to mention that its solutions dictate the large time behavior of  solutions to the Cauchy-Dirichlet problem for the {\it Porous Medium Equation}, i.e.
\begin{equation}
\label{PME}
\left\{\begin{array}{rcll}
\Delta (|u|^{m-1}\,u)&=&u_t,& \text{ in } \Omega\times (0,+\infty),\\
u&=&0,& \text{ on } \partial\Omega\times (0,+\infty),\\
u(\cdot,0)&=& u_0, & \text{ in }\Omega,
\end{array}
\right.
\end{equation}
where 
\[
m=\frac{1}{q-1}.
\]
The reader can see for example \cite[Theorem 3]{AP} and \cite[Theorems 1.1 and 2.1]{Va}.
In this respect, we can say that equation \eqref{LE} plays the same role with respect to \eqref{PME}, as the usual eigenvalue equation does for the heat equation.
\par
Now, it turns out that the question whether $\lambda_1(\Omega;q)$ is isolated or not is tightly connected with the question whether the positive least energy solution of \eqref{LE} is isolated in the set of solutions or not. Again, for a general open bounded sets, this is not true: as above, a counterexample is given by any set with countably many connected components.

\subsection{Main results}

We now present the main results of this paper, by postponing some comments on the assumptions to Remark \ref{oss:assumptions} below.
\begin{teoa}
\label{teo:main}
Let $1<q<2$ and let $\Omega\subset\mathbb{R}^N$ be a $C^1$ open bounded set, with a finite number of connected components. We indicate by $w_{\Omega,q}$ the positive least energy in $\Omega$ of \eqref{LE}, i.e. the unique positive minimizer of the energy
\begin{equation}
\label{energy}
\mathfrak{F}_q(\varphi)=\frac{1}{2}\,\int_\Omega |\nabla \varphi|^2\,dx-\frac{1}{q}\,\int_\Omega |\varphi|^q\,dx.
\end{equation}
Then $w_{\Omega,q}$ is {\rm isolated} in the $L^1(\Omega)$ norm topology, i.e. there exists $\delta>0$ such that the neighborhood
\[
\mathcal{I}_\delta(w_{\Omega,q})=\Big\{\varphi\in\mathcal{D}^{1,2}_0(\Omega)\, :\, \|\varphi-w_{\Omega,q}\|_{L^1(\Omega)}< \delta\Big\},
\]
does not contain any other solution of the Lane-Emden equation.
\end{teoa}
The previous result implies the following one, which has been announced in \cite{BF}.
\begin{teob}
\label{teo:main2}
Let $1<q<2$ and let $\Omega\subset\mathbb{R}^N$ be a $C^1$ open bounded set, with a finite number of connected components. 
Then the first $q-$eigenvalue $\lambda_1(\Omega;q)$
is isolated in $\mathrm{Spec}(\Omega;q)$.
\end{teob}
As a straightforward consequence of Theorem B and of the closedness of $\mathrm{Spec}(\Omega;q)$, we get the following
\begin{coro}[The second $q-$eigenvalue]
Under the assumptions of Theorem \ref{teo:main2}, if we set
\[
\lambda_2(\Omega;q):=\inf\{\lambda \in\mathrm{Spec}(\Omega;q)\, :\, \lambda>\lambda_1(\Omega;q)\},
\]
then we have 
\[
\lambda_1(\Omega;q)<\lambda_2(\Omega;q)\qquad \text{ and }\qquad \lambda_2(\Omega;q)\in\mathrm{Spec}(\Omega;q).
\]
\end{coro}

\begin{oss}
\label{oss:assumptions}
Some comments are in order about our results:
\begin{enumerate}
\item the $C^1$ regularity of $\Omega$ is not really needed, it is placed here just for ease of presentation.  
Indeed, our result is more general, as we can allow for Lipschitz sets (see Theorems \ref{teo:A} \& \ref{teo:B} below). However, in this case, if the Lipschitz constant is too large, then the result is valid for a restricted range
\[
q_\Omega<q<2,
\]
for a limit exponent $1\le q_\Omega<2$ depending on the Lipschitz constant of $\Omega$ and degenerating to $2$ as the Lipschitz constant blows-up (actually, the distinguishing condition is slightly more refined, as it depends only on the ``interior'' angles of the sets, see the discussions in Sections \ref{sec:2} and \ref{sec:5} below);
\vskip.2cm
\item in the statement Theorem \ref{teo:main}, the precision {\it least energy solution} can be omitted when $\Omega$ is connected, since in this case the Lane-Emden equation has a unique positive solution, which is indeed the unique positive minimizer of the associated energy functional. On the contrary, when $\Omega$ has $k$ connected components, the Lane-Emden equation has multiple non-negative solutions (see Remark \ref{oss:multiple}). In this case, positive solutions not having least energy {\it are not isolated}, see Example \ref{exa:next};  
\vskip.2cm
\item finally, the assumption on the number of connected components is optimal, as shown in \cite[Theorem 3.2]{BFcontro}.
\end{enumerate}
\end{oss}

\subsection{Some comments on the case $q>2$}

Our paper is focused on the case $1<q<2$ and our proofs and results do not extend to the super-homogeneous case $2<q<2^*$. The latter is indeed slightly different and some different phenomena may occur. We wish to comment on this case, in order to give a cleaner picture.
\par
We first point out that for $q>2$ the concept of least energy solution is not well-defined, since the relevant energy functional \eqref{energy}
is now {\it unbounded from below}. Moreover, in general it is no more true that the Lane-Emden equation \eqref{LE} has a unique positive least energy solution. Indeed, there has been an extensive study about existence of multiple positive solutions for equation \eqref{LE} in the super-homogeneous regime $2<q<2^*$. We mention \cite[Theorem B]{BC}, \cite{Da1990} and \cite{Da1988}, just to name a few classical results.
\par
The most simple example of this phenomenon is given by a sufficiently thin spherical shell. It is known that for every $2<q<2^*$ there exists a radius $0<r<1$ such that on
\[
A_{r}=\{x\in\mathbb{R}^N\, :\, r<|x|<1\},
\]
any first $q-$eigenfunction (which must have constant sign) {\it is not} radial, see \cite[Proposition 1.2]{Na}. This implies that on $A_r$ 
there exists infinitely many positive solutions of \eqref{LE}, obtained by composing a solution with the group of symmetries of $A_r$. We point out that examples of multiplicity can be exhibited also in presence of a trivial topology, see for example \cite{Da1989} and \cite[Example 4.7]{BF}. But the case of the spherical shell $A_r$ has one interesting feature more: by construction, each positive solution constructed above {\it is not isolated}.
As for the first $q-$eigenvalue, it is not known whether $\lambda_{1}(A_r;q)$ is isolated or not, in this case.
\par
As observed in \cite{Er}, Theorem B holds for $2<q<2^*$ whenever $\lambda_1(\Omega;q)$ is {\it simple}, i.e. there exists a unique first $q-$eigenfunction with unit $L^q$ norm, up to the choice of the sign. However, this condition does not always hold, as exposed above. By \cite[Theorem 4.4]{DGP}, this is known to be true on sets for which a positive first $q-$eigenfunction $u$ is {\it non-degenerate}. This means that the linearized operator
\[
\varphi\mapsto -\Delta \varphi-(q-1)\,u^{q-2}\,\varphi,
\]
does not contain $0$ in its spectrum. Such a condition in general is quite difficult to be checked.
To the best of our knowledge, this has been verified only in some peculiar cases, that we list below:
\begin{itemize}
\item open bounded convex sets in $\mathbb{R}^2$;
\vskip.2cm
\item open bounded sets in $\mathbb{R}^2$, which are convex in the directions $x_1$ and $x_2$ and symmetric about the lines $\{x_1=0\}$ and $\{x_2=0\}$
\vskip.2cm
\item a ball in any dimension $N\ge 2$.
\end{itemize}
The first case is due to Lin (see \cite[Lemma 2]{Lin}), the other two cases are due to Dancer (see \cite[Theorem 5]{Da1988}) and Damascelli, Grossi and Pacella (see \cite[Theorem 4.1]{DGP}). 
\par
For a general open bounded connected set $\Omega$, the best result is due to Lin (see \cite[Lemma 3]{Lin} and also \cite[Proposition 4.3]{BF}): this asserts that there always exists an exponent $q_0=q_0(\Omega)>2$ such that $\lambda_1(\Omega;q)$ is simple (and thus isolated) for $2<q<q_0$. The example of the spherical shell $A_r$ shows that this result is optimal.

\subsection{Strategy of the proof and plan of the paper} The proof of our main result is still based on studying the linearized operator
\begin{equation}
\label{linearized}
\varphi\mapsto -\Delta \varphi-(q-1)\,w_{\Omega,q}^{q-2}\,\varphi,
\end{equation}
around the positive least energy solution \(w_{\Omega, q}\). However, with respect to the case \(2<q<2^*\), in our setting difficulties are reversed. Indeed, by means of a suitable Hardy-type inequality proved in \cite{BFR}, it is quite simple to show that \(0\) does not belong to the spectrum of the operator \eqref{linearized}. On the other hand, since \(1<q<2\) this operator has a singular potential. Thus, rigorously establishing the reduction to the linearized problem requires a careful study of the embedding 
\begin{equation}
\label{embello}
\mathcal{D}^{1,2}_0(\Omega)\hookrightarrow L^2(\Omega;w_{\Omega,q}^{q-2}).
\end{equation}
We now briefly describe the structure of the paper: in Section \ref{sec:2} we recall some definitions and properties of $C^1$ and Lipschitz sets, needed to prove our main results. 
\par
Section \ref{sec:3} deals with some properties of the positive least energy solution $w_{\Omega,q}$ of equation \eqref{LE}. In particular, by appealing to some fine estimates for the Green function of a Lipschitz set obtained in \cite{CMMR}, we obtain an estimate from below on $w_{\Omega,q}$, in terms of a suitable power of the distance from the boundary (see Corollary \ref{coro:fiesta}). 
\par
In Section \ref{sec:4} we prove an abstract version of our main result, namely that the positive least energy solution $w_{\Omega,q}$ is isolated whenever the weighted embedding \ref{embello}
is compact, see Proposition \ref{lm:convergencebis}. This is the cornerstone of Theorem A and Theorem B.
\par
In Section \ref{sec:5}, we analyze the embedding \eqref{embello} and provide some sharp sufficient conditions for this to be compact, see Proposition \ref{lm:weighted} and Corollary \ref{coro:weighted}.
\par
Finally, in Section \ref{sec:6} we join the outcomes of the previous two sections, in order to prove Theorem A and Theorem B, in a larger generality.
\par
The paper is complemented by three appendices, containing: a simple, yet crucial, pointwise inequality; a universal $L^\infty$ bound for solutions of \eqref{LE}; some properties of solutions of \eqref{LE} in convex cones, with an associated study of the embedding \eqref{embello}.

\begin{ack}
Part of this work has been done during some visits of G.\,D.\,P. to the University of Bologna and of L.\,B. \& G.\,D.\,P. to the University of Firenze. The relevant Departments of Mathematics are kindly acknowledged for their hospitality.
\end{ack}

\section{Lipschitz sets}
\label{sec:2}
For an open bounded set $\Omega\subset\mathbb{R}^N$, in what follows we denote by $d_\Omega$ the distance function from the boundary, i.e.
\[
d_\Omega(x)=\min_{y\in\partial\Omega} |x-y|,\qquad \text{ for every } x\in\Omega.
\]
In the sequel, we will need some fine comparison estimates for solutions of elliptic PDEs. These will be taken from \cite{CMMR}. In order to be consistent, we adopt the same definition of {\it Lipschitz sets} as in \cite{CMMR}.
\begin{defi}
Let $r,h>0$ be two positive numbers, then we set
\[
C(r,h)=\{x=(x',x_N)\in\mathbb{R}^N\, :\, |x'|<r\ \text{ and }\ |x_N|<h\}.
\]
An open bounded set $\Omega\subset\mathbb{R}^N$ is called {\it Lipschitz} if for any $x_0\in\partial\Omega$, there exist $r,h>0$ and a Lipschitz function $\varphi:\mathbb{R}^{N-1}\to\mathbb{R}$ such that, up to a rigid movement taking $x_0$ to the origin, we have
\[
\partial \Omega\cap C(r,h)=\{(x',x_N)\, :\, |x'|<r,\, x_N=\varphi(x_1,\dots,x_{N-1})\},
\]
and
\[
\Omega\cap C(r,h)=\{(x',x_N)\, :\, |x'|<r,\, \varphi(x_1,\dots,x_{N-1})<x_N<h\}.
\]
We call an {\it atlas for $\partial\Omega$} any finite collection of cylinders $\{C_k(r_k,h_k)\}_{1\le k\le m}$, with associated Lipschitz maps $\{\varphi_k\}_{1\le k\le m}$, which covers $\partial\Omega$. Then we define the {\it Lipschitz constant of $\Omega$} as
\begin{equation}\label{e:kappa}
\kappa_\Omega=\inf \Big(\max\{\|\nabla \varphi_k\|_{L^\infty}\, :\, 1\le k\le m\}\Big),
\end{equation}
where the infimum is taken over all atlases for $\partial\Omega$.
\end{defi}
\begin{oss}\label{oss:C1}
Similarly, as in \cite{CMMR}, we say that $\Omega$ {\it is of class $C^1$} if the functions $\varphi$ in the previous definition are of class $C^1$. Then, for $\Omega$ of class $C^1$ we have
\[
\kappa_\Omega=0,
\]
see \cite[equation (3.25)]{CMMR}.
\end{oss}

\begin{defi}[Cones]
\label{defi:cones}
For $0\le \beta<1$, we consider the spherical cap
\[
\mathcal{S}(\beta)=\{\omega\in\mathbb{S}^{N-1}\,:\, \beta<\langle \omega,\mathbf{e}_1\rangle\},
\]
where $\mathbf{e}_1=(1,0,\dots,0)$.
We also consider the axially symmetric cone
\[
\Gamma(\beta,R)=\left\{x\in\mathbb{R}^N\, :\, 0<|x|<R \text{ and } \frac{x}{|x|}\in\mathcal{S}(\beta) \right\}.
\]
We indicate by $\lambda(\mathcal{S}(\beta))$ the first Dirichlet eigenvalue of the Laplace-Beltrami operator
on $\mathcal{S}(\beta)$. Then we define $\alpha(\beta)$ to be the positive root of the equation
\[
\alpha(\beta)\,(N-2+\alpha(\beta))=\lambda(\mathcal{S}(\beta)),
\]
i.e.
\begin{equation}
\label{alpha}
\alpha(\beta)=\frac{\sqrt{(N-2)^2+4\,\lambda(\mathcal{S}(\beta))}-(N-2)}{2}>0.
\end{equation}
\end{defi}

Note that the map \(\beta \to \lambda(\mathcal{S}(\beta))\) is  increasing  and thus so it is the map \(\beta \to \alpha(\beta)\). Furthermore, we have $\lambda(\mathcal{S}(0))=N-1$ and an associated eigenfunction is given by $\varphi(x)=x_1$. Correspondingly, we obtain
\[
\alpha(0)=1.
\]
More generally, we have 
\begin{equation}\label{e:asymptotics}
\begin{aligned}
\lambda(\mathcal{S}(\beta))&\searrow N-1  & \text{ and } \ &\alpha(\beta)\searrow 1,  &&\text{as } \beta\searrow 0,
\\
\lambda(\mathcal{S}(\beta))&\nearrow +\infty & \text{ and }\ &\alpha(\beta)\nearrow +\infty, &&\text{as } \beta\nearrow 1. 
\end{aligned}
\end{equation}
In the case $N=2$, it is easily seen that $\lambda(\mathcal{S}(\beta))$ is the first Dirichlet eigenvalue of the operator $\varphi\mapsto -\varphi''$ on the interval
\[
(-\arccos\beta,\arccos\beta).
\]
Thus, in this case we have
\begin{equation}
\label{cono2}
\lambda(\mathcal{S}(\beta))=\left(\frac{\pi}{2\,\arccos\beta}\right)^2\qquad\text{and}\qquad \alpha(\beta)=\frac{\pi}{2\,\arccos\beta}
\end{equation}

Let us now precisely state the inner cone condition that will be used throughout the paper.
\begin{defi}
\label{defi:openingindex}
Let $\Omega\subset\mathbb{R}^N$ be an open bounded set. For a given $R>0$, we indicate 
\[
\Omega_R=\{x\in\Omega\, :\, d_\Omega(x)> R\}.
\]
We say that $\Omega$ satisfies an {\it inner cone condition of index $\beta\in (0,1)$}
if there exists $0<R<\mathrm{diam\,}(\Omega)$ such that 
\begin{equation}
\label{cone}
\forall x\in \overline{\Omega}\setminus \Omega_R\ \text{ $\exists$ an isometry $\mathcal{O}$ of $\mathbb{R}^N$ such that  $\mathcal{O}(x)=0$ and }\ x+\mathcal{O}\,(\Gamma(\beta,R))\subset \Omega.
\end{equation}
\end{defi}
We also define the homogeneity index of a set \(\Omega\) as follows
\begin{defi}
Let $\Omega\subset\mathbb{R}^N$ be an open bounded  set. We define its {\it cone index} as
\[
\beta_\Omega=\inf\bigl\{ \beta\in (0,1)\,:\,   \text{$\Omega$ satisfies an inner cone condition of index $\beta\in (0,1)$}\bigr\},
\]
and its \emph{homogeneity index} as 
\[
\alpha_\Omega=\inf\bigl\{ \alpha(\beta)\,:\,   \beta \ge \beta_{\Omega} \}=\alpha(\beta_\Omega).
\]
\end{defi}

\begin{oss}
It is a classical fact that if \(\Omega\) is an open bounded Lipschitz set, then it satisfies  an inner cone condition with  index $\beta$, for all \(\beta\) such that 
\[
\frac{\kappa_\Omega}{\sqrt{1+\kappa_\Omega^2}}<\beta<1.
\]
Here  $\kappa_\Omega$  is the Lipschitz constant defined in \eqref{e:kappa}. Hence 
\begin{equation}\label{e:upperbound}
\beta_\Omega \le \frac{\kappa_\Omega}{\sqrt{1+\kappa_\Omega^2}}\qquad \text{and}\qquad \alpha_\Omega<+\infty.
\end{equation}
In particular, by Remark  \ref{oss:C1} and \eqref{e:asymptotics}, one has that 
\begin{equation}\label{e:C1}
\text{\(\Omega\)  of class \(C^1\)} \qquad\Longrightarrow \qquad \beta_\Omega=0\quad \text{ and }\quad \alpha_\Omega=1.
\end{equation}
Note however that for a Lipschitz set the inequality in \eqref{e:upperbound} can be strict. This is the case, for example, when the set have ``concave'' corners, see Figure \ref{fig:1}. 
\end{oss}
\begin{figure}
\includegraphics[scale=.4]{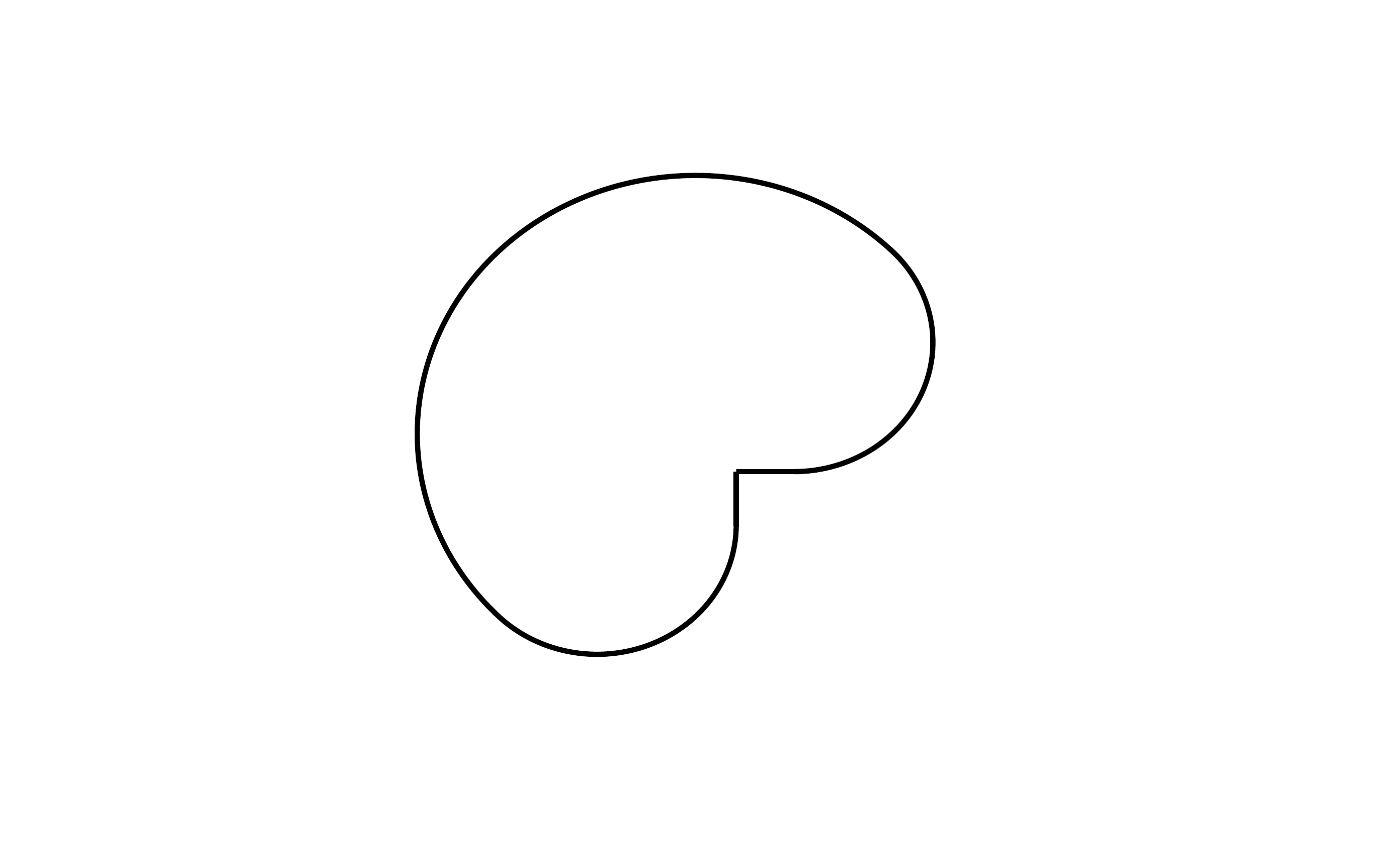}
\caption{A set with Lipschitz constant $\kappa_\Omega=1$, which satisfies $\beta_\Omega=0$.}
\label{fig:1}
\end{figure}

\section{Least energy solutions}
\label{sec:3}

\begin{defi}
\label{defi:boundeddensity}
Let $\Omega\subset\mathbb{R}^N$ be an open bounded set. For $1< q<2$, we define $w_{\Omega,q}$ to be the unique solution of 
\[
\min_{\varphi\in \mathcal{D}^{1,2}_0(\Omega)}\left\{\frac{1}{2}\,\int_\Omega |\nabla \varphi|^2\,dx-\frac{1}{q}\,\int_\Omega \varphi^q\,dx\, :\, \varphi\ge 0 \text{ a.\,e. in }\Omega\right\}.
\]
Existence follows by using the Direct Methods in the Calculus of Variations, while for uniqueness we refer for example to \cite[Lemma 2.2]{BFR}.
\end{defi}
\begin{oss}
\label{oss:multiple}
The function $w_{\Omega,q}$ is the unique non-negative solution of 
\[
\min_{\varphi\in \mathcal{D}^{1,2}_0(\Omega)}\left\{\frac{1}{2}\,\int_\Omega |\nabla \varphi|^2\,dx-\frac{1}{q}\,\int_\Omega |\varphi|^q\,dx\right\},
\]
as well. Consequently, when $\Omega\subset\mathbb{R}^N$ is connected, it is the unique non-negative solution of the Lane-Emden equation \eqref{LE}.
Then, it is not difficult to see that
\[
\lambda_1(\Omega;q)=\|w_{\Omega,q}\|_{L^q(\Omega)}^{q-2},
\]
and the rescaled function
\[
u=\frac{w_{\Omega,q}}{\|w_{\Omega,q}\|_{L^q(\Omega)}},
\]
is the (unique) first positive $q-$eigenfunction of $\Omega$, with unit $L^q$ norm.
\par
On the contrary, when $\Omega$ is disconnected, equation \eqref{LE} has many non-negative solutions. More precisely, if $\Omega$ has $k$ connected components $\Omega_1,\dots,\Omega_k$, then \eqref{LE} has exactly $2^k-1$ positive solutions, given by
\[
\delta_1\,w_{\Omega_1,q}+\dots+\delta_k\,w_{\Omega_k,q},
\]
where each $\delta_i\in\{0,1\}$ and they are not all equal to $0$. In this case, the function $w_{\Omega,q}$ is given by
\begin{equation}
\label{sommine}
w_{\Omega,q}=\sum_{i=1}^k w_{\Omega_i,q},
\end{equation}
and it gives the positive least energy solution of \eqref{LE}.
Accordingly, in this case the first $q-$eigenvalue of $\Omega$ is given by (see \cite[Corollary 2.6]{BFcontro})
\[
\lambda_1(\Omega;q)=\left[\sum_{i=1}^k \left(\frac{1}{\lambda_1(\Omega_i;q)}\right)^\frac{q}{2-q}\right]^\frac{q-2}{q},
\]
and it is not simple.
\end{oss}
The following Hardy-type inequality is a particular case of a general result proved in \cite{BFR}. This simpler result is enough for our purposes and it is obtained by taking $\delta=1$ in \cite[Theorem 3.1 \& Corollary 3.4]{BFR}.
\begin{prop}[Hardy-Lane-Emden inequality]
\label{lm:HLE}
Let $1<q<2$ and let $\Omega\subset\mathbb{R}^N$ be an open bounded set. Then for every $\varphi\in \mathcal{D}^{1,2}_0(\Omega)$, we have
\[
\int_\Omega w_{\Omega,q}^{q-2}\,|\varphi|^2\,dx\le \int_\Omega |\nabla \varphi|^2\,dx.
\]
In particular, the embedding 
\[
\mathcal{D}^{1,2}_0(\Omega)\hookrightarrow L^2(\Omega; w_{\Omega,q}^{q-2}),
\]
is continuous. 
\end{prop}

Given an open set $\Omega\subset\mathbb{R}^N$, we denote by $\mathcal{G}_\Omega$ the Green function for the Dirichlet-Laplacian on $\Omega$. We also recall the notation
\[
\Omega_R=\{x\in\Omega\, :\, d_\Omega(x)> R\}.
\]
The following result is contained in \cite[Proposition 3.6]{CMMR}. A similar result, under slightly stronger assumptions and with a worse control on the constants, was previously obtained in \cite{MS}.
\begin{prop}
Let $\Omega\subset\mathbb{R}^N$ be an open bounded Lipschitz set which satisfies an inner cone condition of index \(\beta\in (0,1)\). Let \(R>0\) be such that  \eqref{cone} hold true. Then, for \(\alpha=\alpha(\beta)\) defined in \eqref{alpha}, we have
\begin{equation}
\label{salvaculo}
\mathcal{G}_\Omega(x,y)\ge \frac{1}{C}\,d_\Omega(x)^{\alpha},\qquad \text{ for every } x\in \Omega \text{ and } y\in \Omega_R,
\end{equation}
where the constant $C>0$ depends on $N, \Omega$ and \(R\) only.
\end{prop}
From the previous result, we immediately get the following one. This is an essential ingredient for the proof of our main result.
\begin{coro}
\label{coro:fiesta}
Let $\Omega\subset\mathbb{R}^N$ be an open bounded Lipschitz set, with homogeneity index $\alpha_\Omega$.
For $1<q<2$, let \(w_{\Omega,q}\) be the function in Definition \ref{defi:boundeddensity}.
Then, for all \(\alpha>\alpha_\Omega\)  we have
\begin{equation}\label{e:culo}
w_{\Omega,q}(x)\ge \frac{1}{\mathcal{C}}\,d_\Omega(x)^{\alpha},\qquad \text{ for every } x\in\Omega,
\end{equation}
where the constant $\mathcal{C}>0$ depends on $N, q, \Omega$ and \(\alpha\) (and it might  degenerate as $\alpha$ converges to $\alpha_\Omega$).
\end{coro}
\begin{proof}
The function $w_{\Omega,q}$ solves
\[
-\Delta w_{q,\Omega}=w_{\Omega,q}^{q-1},\qquad \text{ in } \Omega,
\]
with homogeneous Dirichlet conditions.
Thus, by the representation formula for Poisson's equation, we have
\[
w_{\Omega,q}(x)=\int_\Omega \mathcal{G}_\Omega(x,y)\,w_{\Omega,q}(y)^{q-1}\,dy,\qquad \text{ for } x\in\Omega.
\]
We fix \(\alpha >\alpha_\Omega\) and we let \(\beta>\beta_\Omega\) be such that \(\alpha=\alpha(\beta)\). By definition of cone index $\beta_\Omega$, we have that \(\Omega\) satisfies an inner cone condition of index \(\beta\). We can thus  use  \eqref{salvaculo} to infer
\begin{equation}
\label{struttura}
\begin{split}
w_{\Omega,q}(x)\ge \int_{\Omega_R}\mathcal{G}_\Omega(x,y)\,w_{\Omega,q}(y)^{q-1}\,dy&\ge \frac{1}{C}\, \int_{\Omega_R} d_\Omega(x)^{\alpha}\,w_{\Omega,q}(y)^{q-1}\,dy\\
& \ge\frac{1}{C}\,\left(\min_{\Omega_R} w_{\Omega,q}\right)^{q-1}\,|\Omega_R|\,d_\Omega(x)^{\alpha}.
\end{split}
\end{equation}
Observe that $\min_{\Omega_R} w_{\Omega,q}>0$, by the minimum principle. We can now use the lack of homogeneity of the equation to improve the previous estimate: by passing to the minimum over $\Omega_R$ in \eqref{struttura}, we get
\[
\min_{\Omega_R} w_{\Omega,q}\ge \frac{1}{C}\,\left(\min_{\Omega_R} w_{\Omega,q}\right)^{q-1}\,|\Omega_R|\,R^{\alpha},
\]
that is
\[
\min_{\Omega_R} w_{\Omega,q}\ge \left(\frac{1}{C}\,|\Omega_R|\,R^{\alpha}\right)^\frac{1}{2-q}.
\]
By spending this information in \eqref{struttura}, we finally get
\[
w_{\Omega,q}(x)\ge \frac{1}{C}\,\left(\frac{1}{C}\,|\Omega_R|\,R^{\alpha}\right)^\frac{q-1}{2-q}\,|\Omega_R|\,d_\Omega(x)^{\alpha},
\]
as desired.
\end{proof}

\section{An abstract result}
\label{sec:4}

The following result is the key ingredient for the proof of our main result. 
\begin{prop}
\label{lm:convergencebis}
Let $1<q<2$ and let $\Omega\subset\mathbb{R}^N$ be an open bounded set.
Let us assume that 
\begin{equation}
\label{ipotesicroce}
\text{the embedding }\mathcal{D}^{1,2}_0(\Omega)\hookrightarrow L^2(\Omega; w_{\Omega,q}^{q-2}) \text{ is compact}.
\end{equation}
Then there exists $\delta>0$ such that the neighborhood
\[
\mathcal{I}_\delta(w_{\Omega,q})=\Big\{\varphi\in \mathcal{D}^{1,2}_0(\Omega)\, :\, \|\varphi-w_{\Omega,q}\|_{L^1(\Omega)}<\delta\Big\},
\]
does not contain any solution of the Lane-Emden equation \eqref{LE}.
\end{prop}
\begin{proof}
We argue by contradiction. We assume that there exists a sequence $\{U_n\}_{n\in\mathbb{N}}\subset\mathcal{D}^{1,2}_0(\Omega)$ of solutions  of the equation \eqref{LE}, such that 
\[
\lim_{n\to\infty} \|U_n-w_{\Omega,q}\|_{L^1(\Omega)}=0.
\]
We first observe that by Corollary \ref{coro:accesso}, we automatically get
\[
\lim_{n\to\infty} \|\nabla w_{\Omega,q}-\nabla U_n\|_{L^2(\Omega)}=0,
\]
as well.
\par
By subtracting the equations satisfied by $w_{\Omega,q}$ and $U_n$, we get
\begin{equation}
\label{equadiff}
\int_\Omega \langle \nabla(w_{\Omega,q}-U_n),\nabla \varphi\rangle\,dx=\int_\Omega (w_{\Omega,q}^{q-1}-|U_n|^{q-2}\,U_n)\,\varphi\,dx,
\end{equation}
for every $\varphi\in \mathcal{D}^{1,2}_0(\Omega)$. By introducing the function
\[
W_n=\frac{w_{\Omega,q}^{q-1}-|U_n|^{q-2}\,U_n}{w_{\Omega,q}-U_n},
\]
\eqref{equadiff} can be rewritten as
\[
\int_\Omega \langle \nabla(w_{\Omega,q}-U_n),\nabla \varphi\rangle\,dx=\int_\Omega W_n\,(w_{\Omega,q}-U_n)\,\varphi\,dx.
\]
Thus, if we define the rescaled sequence
\[
\phi_n=\frac{w_{\Omega,q}-U_n}{\|w_{\Omega,q}-U_n\|_{L^2(\Omega;w_{\Omega,q}^{q-2})}},\qquad \text{ for } n\in\mathbb{N},
\]
we get that $\phi_n\in\mathcal{D}^{1,2}_0(\Omega)$ is a weak solution of
\begin{equation}
\label{equazione}
-\Delta \phi_n=W_n\,\phi_n.
\end{equation}
We observe that 
\begin{equation}
\label{puntuale}
0\le W_n(x)\le 2^{2-q}\,w_{\Omega,q}^{q-2}(x), \qquad \text{ for every } x\in\Omega,
\end{equation}
thanks to Lemma \ref{lm:quasilipschitz}. 
Then, by testing the equation \eqref{equazione} with $\phi_n$ itself and recalling the normalization taken, we get
\[
\int_\Omega |\nabla \phi_n|^2\,dx\le 2^{2-q}\,\int_\Omega w_{\Omega,q}^{q-2}\,|\phi_n|^2\,dx=2^{2-q},\qquad \text{ for every } n\in\mathbb{N}.
\]
Thus, the assumption on the compactness of the embedding $\mathcal{D}^{1,2}_0(\Omega) \hookrightarrow L^2(\Omega;w_\Omega^{q-2})$ implies that $\{\phi_n\}_{n\in\mathbb{N}}$ converges (up to a subsequence) weakly in $\mathcal{D}^{1,2}_0(\Omega)$ and strongly in $L^2(\Omega;w_{\Omega,q}^{q-2})$ to $\phi\in \mathcal{D}^{1,2}_0(\Omega)$. In particular, we still have
\begin{equation}
\label{normauno}
\|\phi\|_{L^2(\Omega;w_{\Omega,q}^{q-2})}=1,\qquad \text{ so that }\ \phi\not\equiv 0.
\end{equation}
Observe that we have\footnote{This can be seen as follows: take $\Omega'\Subset\Omega$, thus $U\ge 1/C_{\Omega'}$ on $\Omega'$. By using the equation and standard Elliptic Regularity, we have that $U_n$ converges uniformly on $\overline{\Omega'}$ to $U$. Thus for $n$ large enough the function
\[
t\mapsto \Big(|U(x)+t\,(U_n(x)-U(x))|^{q-2}\,(U(x)+t\,(U_n(x)-U(x)))\Big),\qquad t\in [0,1],
\]
is differentiable, for every $x\in\Omega'$.
Then we write
\[
\begin{split}
U^{q-1}-|U_n|^{q-2}\,U_n&=-\int_0^1 \frac{d}{dt} \Big(|U+t\,(U_n-U)|^{q-2}\,(U+t\,(U_n-U))\Big)\,dt\\
&=(q-1)\,\left(\int_0^1 |U+t\,(U_n-U)|^{q-2}\,dt\right)\,(U-U_n).
\end{split}
\]
By dividing by $(U-U_n)$ and passing to the limit, we get the desired conclusion for every $x\in\Omega'$. By arbitrariness of $\Omega'\Subset\Omega$, we get the pointwise almost everywhere convergence.}
\begin{equation}
\label{ptw}
\lim_{n\to\infty} W_n(x)=(q-1)\,w_{\Omega,q}^{q-2}(x),\qquad \text{ for a.\,e. }x\in\Omega.
\end{equation}
By testing equation \eqref{equazione} with $\phi_n$ itself, we get
\[
\int_\Omega |\nabla \phi_n|^2\,dx=\int_\Omega W_n\,|\phi_n|^2\,dx\le \int_\Omega W_n\,\Big||\phi_n|^2-|\phi|^2\Big|\,dx+\int_\Omega W_n\,|\phi|^2\,dx.
\]
By using the lower semicontinuity of the $L^2$ norm in the left-hand side and the Dominated Convergence Theorem in the right-hand side,
in view of \eqref{ptw} and \eqref{normauno} we get
\begin{equation}
\label{intertempo}
\int_\Omega |\nabla \phi|^2\,dx\le \limsup_{n\to\infty}\int_\Omega W_n\,\Big||\phi_n|^2-|\phi|^2\Big|\,dx+(q-1).
\end{equation}
We are only left with handling the limit 
\[
\limsup_{n\to\infty}\int_\Omega W_n\,\Big||\phi_n|^2-|\phi|^2\Big|\,dx.
\]
This is done as follows: by elementary manipulations and \eqref{puntuale}
\[
\begin{split}
\int_\Omega W_n\,\Big||\phi_n|^2-|\phi|^2\Big|\,dx&=\int_\Omega W_n\,\Big|\phi_n-\phi\Big|\,\Big|\phi_n+\phi\Big|\,dx\\
&\le 2^{2-q}\,\int_\Omega w_{\Omega,q}^{q-2}\, \Big|\phi_n-\phi\Big|\,\Big|\phi_n+\phi\Big|\,dx\\
&\le 2^{2-q}\,\left(\int_\Omega w_{\Omega,q}^{q-2}\,|\phi_n-\phi|^2\,dx\right)^\frac{1}{2}\\
&\times \left(\int_\Omega w_{\Omega,q}^{q-2}\,|\phi_n+\phi|^2\,dx\right)^\frac{1}{2}.
\end{split}
\]
By using the strong convergence in $L^2(\Omega;w_{\Omega,q}^{q-2})$ for the first integral and the Hardy-Lane-Emden inequality (i.e. Proposition \ref{lm:HLE}) to bound the second integral, we get 
\[
\limsup_{n\to\infty}\int_\Omega W_n\,\Big||\phi_n|^2-|\phi|^2\Big|\,dx=0.
\]
By using this in \eqref{intertempo}, we finally end up with 
\[
\int_\Omega |\nabla \phi|^2\,dx\le (q-1).
\]
On the other hand, by Proposition \ref{lm:HLE} we must have
\[
1=\int_\Omega w_{\Omega}^{q-2}\,|\phi^2|\,dx\le \int_\Omega |\nabla \phi|^2\,dx.
\]
Since $q<2$, the last two displays give the desired contradiction.
\end{proof}

We now proceed to analyze the situation for the first $q-$eigenvalue. Before doing this, we need the following result. 
\begin{lm}
\label{lm:componenti}
Let $1<q<2$ and let $\Omega\subset\mathbb{R}^N$ be an open bounded set, such that the embedding $\mathcal{D}^{1,2}_0(\Omega)\hookrightarrow L^2(\Omega; w_{\Omega,q}^{q-2})$ is compact. Then $\Omega$ must have a finite number of connected components.
\end{lm}
\begin{proof}
We argue by contradiction and assume that $\Omega$ has countably many connected components $\{\Omega_n\}_{n\in\mathbb{N}}$. We take the sequence 
\[
\varphi_n=\frac{w_{\Omega_n,q}}{\|\nabla w_{\Omega_n,q}\|_{L^2(\Omega_n)}},\qquad n\in\mathbb{N},
\]
which is bounded in $\mathcal{D}^{1,2}_0(\Omega)$. By using Poincar\'e inequality, we have 
\begin{equation}
\label{ciaone}
\int_\Omega |\varphi_n|^2\,dx=\int_{\Omega_n} |\varphi_n|^2\,dx\le \frac{1}{\lambda_1(\Omega_n)}\,\int_{\Omega_n} |\nabla \varphi_n|^2\,dx=\frac{1}{\lambda_1(\Omega_n)},
\end{equation}
where we recall that $\lambda_1$ stands for the first eigenvalue of the Dirichlet-Laplacian. 
We now observe that by the {\it Faber-Krahn inequality}, we have 
\[
\lambda_1(\Omega_n)\ge |\Omega_n|^{-\frac{2}{N}}\,\Big(|B|^\frac{2}{N}\,\lambda_1(B)\Big),
\]
where $B$ is any $N-$dimensional ball. Moreover, since $\Omega$ is bounded, we have 
\[
\lim_{n\to\infty} |\Omega_n|=0.
\]
The last two displays show that $\lambda_1(\Omega_n)$ diverges to $+\infty$, as $n$ goes to $\infty$. By \eqref{ciaone} we thus get that $\varphi_n$ converges strongly to $0$ in $L^2(\Omega)$.
\par
On the other hand, by computing the weighted $L^2$ norm of $\varphi_n$ and recalling \eqref{sommine}, we have
\[
\int_\Omega \frac{|\varphi_n|^2}{w_{\Omega,q}^{2-q}}\,dx=\int_{\Omega_n} \frac{|\varphi_n|^2}{w_{\Omega_n,q}^{2-q}}\,dx=\frac{\displaystyle\int_{\Omega_n} w_{\Omega_n,q}^q\,dx}{\displaystyle\int_{\Omega_n} |\nabla w_{\Omega_n,q}|^2\,dx}=1,
\]
where in the last identity we used that $w_{\Omega_n,q}$ solves the Lane-Emden equation on $\Omega_n$. This contradicts the compactness of the embedding $\mathcal{D}^{1,2}_0(\Omega)\hookrightarrow L^2(\Omega; w_{\Omega,q}^{q-2})$.
\end{proof}
By using Proposition \ref{lm:convergencebis}, we can now get the following result for the first $q-$eigenvalue.
\begin{prop}
\label{prop:bis}
Let $1<q<2$ and let $\Omega\subset\mathbb{R}^N$ be an open bounded set. Let us assume that condition \eqref{ipotesicroce} holds.
Then the first $q-$eigenvalue 
\[
\lambda_1(\Omega;q)=\min_{u\in\mathcal{D}^{1,2}_0(\Omega)} \left\{\int_\Omega |\nabla u|^2\,dx\, :\, \int_\Omega |u|^q\,dx=1\right\},
\]
is isolated in $\mathrm{Spec}(\Omega;q)$.
\end{prop}
\begin{proof}
We argue by contradiction and assume that there exists a sequence $\{\lambda_n\}_{n\in\mathbb{N}}\subset\mathrm{Spec}(\Omega;q)$ such that
\begin{equation}
\label{assurdo}
\lim_{n\to\infty} \lambda_n=\lambda_1(\Omega;q).
\end{equation}
We have to consider two different cases: either $\Omega$ is connected or not.
\vskip.2cm\noindent
{\bf Case 1: $\Omega$ is connected}. 
There exists a sequence $\{u_n\}_{n\in\mathbb{N}}\subset\mathcal{D}^{1,2}_0(\Omega)$ of normalized $q-$eigenfunctions for $\Omega$, i.\,e.
\[
-\Delta u_n=\lambda_n\,|u_n|^{q-2}\,u_n, \text{ in }\Omega,\qquad \text{ with }\ \|u_n\|_{L^q(\Omega)}=1.
\]
By using the equation, it is not difficult to see that
\begin{equation}
\label{assurdo2}
\lim_{n\to\infty} \|\nabla u_n-\nabla u\|_{L^2(\Omega)}=0,
\end{equation}
where $u\in \mathcal{D}^{1,2}_0(\Omega)$ is a first $q-$eigenfunction, with unit $L^q$ norm. Since we are assuming $\Omega$ to be connected, the first $q-$eigenvalue is simple (see \cite[Theorem 3.1]{BF}). Thus, upon replacing $u_n$ with $-u_n$, we can suppose that $u$ is the first positive $q-$eigenfunction, with unit $L^q$ norm.
\par
We note that the rescaled functions
\[
U=\lambda_1(\Omega;q)^{\frac{1}{q-2}}\,u\qquad \text{ and }\qquad U_n=\lambda_n^{\frac{1}{q-2}}\,u_n,
\]
solve the Lane-Emden equation \eqref{LE} in $\Omega$. Moreover, by uniqueness of the positive solution, $U$ coincides with $w_{\Omega,q}$.
From \eqref{assurdo} and \eqref{assurdo2}, we can thus infer that
\[
\lim_{n\to\infty} \|\nabla U_n-\nabla w_{\Omega,q}\|_{L^2(\Omega)}=0.
\]
However, this contradicts Proposition \ref{lm:convergencebis}.
\vskip.2cm\noindent
{\bf Case 2: $\Omega$ is not connected.} By Lemma \ref{lm:componenti}, we know that assumption \eqref{ipotesicroce} guarantees that $\Omega$ has a finite number $\Omega_1,\dots,\Omega_k$ of connected components. 
By the ``spin formula'' of \cite[Proposition 2.1 \& Corollary 2.2]{BFcontro}, we know that 
\[
\lambda_n=\left[\sum_{i=1}^k \left(\frac{\delta_{n,i}}{\lambda_{n,i}}\right)^\frac{q}{2-q}\right]^\frac{q-2}{q},
\] 
for some $\lambda_{n,i}\in\mathrm{Spec}(\Omega_i;q)$ and $\delta_{n,i}\in \{0,1\}$ such that
\[
\sum_{i=1}^k \delta_{n,i}\not =0,\qquad \text{ for every } n\in\mathbb{N}.
\]
Moreover,  by \cite[Corollary 2.6]{BFcontro}, we must have
\[
\lambda_1(\Omega;q)=\left[\sum_{i=1}^k \left(\frac{1}{\lambda_1(\Omega_i;q)}\right)^\frac{q}{2-q}\right]^\frac{q-2}{q}.
\]
In light of \eqref{assurdo}, we thus get that\footnote{Observe that the function
\[
f(t_1,\dots,t_k)=\left[\sum_{i=1}^k t_i^\frac{q}{2-q}\right]^\frac{q-2}{q},\qquad \text{ for } t_i\le \frac{1}{\lambda_1(\Omega_i;q)},
\]
is strictly decreasing in each argument and it uniquely attains its minimum when 
\[
t_i= \frac{1}{\lambda_1(\Omega_i;q)},\qquad \text{ for every } i=1,\dots,k.
\]
Thus \eqref{assurdo} entails that
\[
\lim_{n\to\infty}\frac{\delta_{n,i}}{\lambda_{n,i}}=\frac{1}{\lambda_1(\Omega_i;q)},\qquad \text{ for every } i=1,\dots,k.
\]}
\[
\lambda_n=\left[\sum_{i=1}^k \left(\frac{1}{\lambda_{n,i}}\right)^\frac{q}{2-q}\right]^\frac{q-2}{q},
\]
for $n$ large enough, with $\lambda_{n,i}$ converging to $\lambda_1(\Omega_i;q)$, for every $i=1,\dots,k$. However, this contradicts the fact that each $\lambda_1(\Omega_i;q)$ is isolated, by the first part of the proof.
\end{proof}

\section{A weighted embedding}
\label{sec:5}

The results of the previous section lead us to study conditions on $\Omega$ under which
\[
\text{the embedding }\mathcal{D}^{1,2}_0(\Omega)\hookrightarrow L^2(\Omega; w_{\Omega,q}^{q-2}) \text{ is compact}.
\]
We have seen in Lemma \ref{lm:componenti} that a necessary condition is that $\Omega$ has a finite number of connected components.
We now provide a sufficient condition, as well. The sharpness of the assumptions is discussed in Example \ref{exa:viola} below.
\begin{prop}
\label{lm:weighted}
Let $\Omega\subset\mathbb{R}^N$ be an open bounded Lipschitz set, with homogeneity index $\alpha_\Omega$. Then:
\begin{itemize}
\item if $1\le \alpha_\Omega\le 2$, the embedding $\mathcal{D}^{1,2}_0(\Omega)\hookrightarrow L^2(\Omega; w_{\Omega,q}^{q-2})$ is compact for every $1<q<2$;
\vskip.2cm
\item if $\alpha_\Omega>2$, the embedding $\mathcal{D}^{1,2}_0(\Omega)\hookrightarrow L^2(\Omega; w_{\Omega,q}^{q-2})$ is compact for every
\begin{equation*}
2-\frac{2}{\alpha_\Omega}<q<2.
\end{equation*}
In particular, by  \eqref{e:C1}, the embedding \(\mathcal{D}^{1,2}_0(\Omega)\hookrightarrow L^2(\Omega; w_{\Omega,q}^{q-2})\) is compact for every $1<q<2$ if \(\Omega\) is of class \(C^1\).
\end{itemize}
\end{prop}
\begin{proof}
Since $\Omega\subset\mathbb{R}^N$ is an open bounded Lipschitz set, there exists a constant $C_\Omega>0$ such that the classical {\it Hardy inequality}  holds
\begin{equation}
\label{necas}
\int_\Omega \frac{|\varphi|^2}{d_\Omega^2}\,dx\le C_\Omega\,\int_\Omega |\nabla \varphi|^2\,dx,\qquad \text{ for every } \varphi\in\mathcal{D}^{1,2}_0(\Omega),
\end{equation}
see \cite[Th\'eor\`eme 1.6]{Ne} or also \cite[Theorem 8.4]{Ku}. We now discuss separately the two cases:
\vskip.2cm\noindent
{\bf Case $\alpha_\Omega\le 2$}. In this case, we have 
\[
\alpha_\Omega\le 2<\frac{2}{2-q},\qquad \text{ for every } 1<q<2.
\]
Thus, we can fix \(\alpha_\Omega<\alpha <2/(2-q)\) such that \eqref{e:culo} holds. By   \eqref{necas} and H\"older's inequality with exponents
\[
\frac{2}{(2-q)\,\alpha}\qquad \text{ and }\qquad \frac{2}{2-(2-q)\,\alpha},
\] 
we get for every $\varphi\in \mathcal{D}^{1,2}_0(\Omega)$
\[
\begin{split}
\int_\Omega \frac{|\varphi|^2}{d_\Omega^{(2-q)\,\alpha}}\,dx&\le \left(\int_\Omega \frac{|\varphi|^2}{d_\Omega^2}\right)^\frac{(2-q)\,\alpha}{2}\,\left(\int_\Omega |\varphi|^2\,dx\right)^{1-\frac{2-q}{2}\,\alpha}\\
&\le C_\Omega^\frac{(2-q)\,\alpha}{2}\,\left(\int_\Omega |\nabla \varphi|^2\,dx\right)^\frac{(2-q)\,\alpha}{2}\,\left(\int_\Omega |\varphi|^2\,dx\right)^{1-\frac{2-q}{2}\,\alpha}.
\end{split}
\]
Moreover, by \eqref{e:culo} we have $d_\Omega^{\alpha}\le \mathcal{C}\, w_{\Omega,q}$. Thus we get the following interpolation inequality
\[
\int_\Omega \frac{|\varphi|^2}{w_{\Omega,q}^{2-q}}\,dx\le \mathcal{C}^{2-q}\,C_\Omega^\frac{(2-q)\,\alpha}{2}\,\left(\int_\Omega |\nabla \varphi|^2\,dx\right)^\frac{(2-q)\,\alpha}{2}\,\left(\int_\Omega |\varphi|^2\,dx\right)^{1-\frac{2-q}{2}\,\alpha}.
\]
By using this and recalling that the embedding $\mathcal{D}^{1,2}_0(\Omega)\hookrightarrow L^2(\Omega)$ is compact for an open bounded set, we get the conclusion. 
\vskip.2cm\noindent
{\bf Case $\alpha_\Omega> 2$}. In this case, we have 
\[
\alpha_\Omega<\frac{2}{2-q},\qquad \text{ for every  $q$ such that } 2-\frac{2}{\alpha_\Omega}<q<2.
\]
We can repeat the previous argument and get again the desired conclusion.
\end{proof}
By the very definition of $\alpha_\Omega$, the assumption $\alpha_\Omega\le 2$ means that the ``corners'' of $\Omega$ should not be ``too narrow''. On the other hand, when $\alpha_\Omega>2$, we have that the higher the value of $\alpha_\Omega$ is, the smaller is the set of exponents $q$ for which the relevant embedding is compact.
\par
Once again, the two-dimensional case is easier to understand. For $N=2$, we can reformulate the previous result as follows
\begin{coro}[Two-dimensional case]
\label{coro:weighted}
Let $\Omega\subset\mathbb{R}^2$ be an open bounded Lipschitz set, with cone index $\beta_\Omega\in [0,1)$.
Then:
\begin{itemize}
\item if $\beta_\Omega\le \cos(\pi/4)$, the embedding $\mathcal{D}^{1,2}_0(\Omega)\hookrightarrow L^2(\Omega; w_{\Omega,q}^{q-2})$ is compact for every $1<q<2$;
\vskip.2cm
\item if $\beta_\Omega>\cos(\pi/4)$, the embedding $\mathcal{D}^{1,2}_0(\Omega)\hookrightarrow L^2(\Omega; w_{\Omega,q}^{q-2})$ is compact for every
\begin{equation*}
2-\frac{4}{\pi}\,\arccos(\beta_\Omega)<q<2.
\end{equation*}
\end{itemize}
\end{coro}
\begin{proof}
By \eqref{cono2} we know that
\[
\alpha_\Omega=\frac{\pi}{2\,\arccos(\beta_\Omega)}.
\]
From this we get 
\[
\alpha_\Omega\le 2\qquad \Longleftrightarrow\qquad \arccos\beta\ge \frac{\pi}{4}\qquad \Longleftrightarrow\qquad \beta\le \cos\left(\frac{\pi}{4}\right),
\]
and thus the conclusion follows from Proposition \ref{lm:weighted}.
\end{proof}
The assumptions in Proposition \ref{lm:weighted} are sharp. Indeed, when these are not in force, the compactness of the embedding can badly fail, even among convex sets.
Indeed, we can produce an open bounded convex set $\Omega\subset\mathbb{R}^N$ such that
\begin{itemize}
\item $\alpha_\Omega>2$;
\vskip.2cm
\item for every $1<q<2-2/\alpha_\Omega$,
the embedding $\mathcal{D}^{1,2}_0(\Omega)\hookrightarrow L^2(\Omega; w_{\Omega,q}^{q-2})$ is not compact.
\end{itemize}
This is shown in the following
\begin{exa}
\label{exa:viola}
We used the same notations of Definition \ref{defi:cones}.
Let $\beta>0$ be such that
\begin{equation}
\label{beta0}
2\,N<\lambda(\mathcal{S}(\beta)).
\end{equation}
For every $R>0$, we consider the convex cone $\Gamma(\beta,R)$. We observe that the function 
\[
\Phi(t)=t\,(N-2+t),\qquad \text{ for }t\ge 0.
\] 
is monotone increasing and $\Phi(2)=2\,N$. By recalling the definition of $\alpha_{\Gamma(\beta,R)}$, this implies that
\[
\text{ condition }\eqref{beta0}\qquad \Longleftrightarrow\qquad \alpha_{\Gamma(\beta,R)}>2.
\]
We show that for every 
\begin{equation}
\label{q}
1<q< 2-\frac{2}{\alpha_{\Gamma(\beta,R)}},
\end{equation}
the embedding
\[
\mathcal{D}^{1,2}_0(\Gamma(\beta,R))\hookrightarrow L^2(\Gamma(\beta,R);w_{\Gamma(\beta,R),q}^{q-2}),
\] 
is not compact. 
In order to prove this, it is sufficient to observe that with our choice \eqref{q} we have
\[
\frac{2}{2-q}<\alpha_{\Gamma(\beta,R)}\qquad \text{ which implies }\qquad \Phi\left(\frac{2}{2-q}\right)<\Phi(\alpha_{\Gamma(\beta,R)})=\lambda(\mathcal{S}(\beta)).
\]
Thus the claimed assertion follows from Proposition \ref{prop:trac} below.
\end{exa}

\section{Proofs of the main results}
\label{sec:6}

\subsection{Proofs}
By combining the compactness result of Proposition \ref{lm:weighted} with Proposition \ref{lm:convergencebis}, we now get the following more general version of Theorem A. We recall that $\alpha_\Omega$ is homogeneity index of $\Omega$, defined in Definition \ref{defi:openingindex}.
\begin{teo}
\label{teo:A}
Let $\Omega\subset\mathbb{R}^N$ be an open bounded Lipschitz set, with homogeneity index $\alpha_\Omega$. Let us suppose that $\Omega$ has a finite number of connected components. If we set
\begin{equation}
\label{qomega}
q_\Omega:=\max\left\{2-\frac{2}{\alpha_\Omega},1\right\},
\end{equation}
then, for every $q_\Omega<q<2$,
the positive least energy solution $w_{\Omega,q}\in\mathcal{D}^{1,2}_0(\Omega)$ of equation \eqref{LE}
is isolated in the $L^1(\Omega)$ norm topology, i.e. there exists $\delta>0$ such that the neighborhood
\[
\mathcal{I}_\delta(w_{\Omega,q})=\Big\{\varphi\in\mathcal{D}^{1,2}_0(\Omega)\, :\, \|\varphi-w_{\Omega,q}\|_{L^1(\Omega)}< \delta\Big\},
\]
does not contain any other solution of the Lane-Emden equation.
\end{teo}
\begin{oss}
We recall that for a $C^1$ set, we have $\alpha_\Omega=1$. Thus in this case from \eqref{qomega} we get
\[
q_\Omega=1,
\]
and we recover the statement of Theorem A. More generally, observe that $q_\Omega=1$ whenever $\alpha_\Omega\le 2$.
\end{oss}
As for the claimed isolation result of Theorem B, this follows by combining Proposition \ref{lm:weighted} with Proposition \ref{prop:bis}. The final outcome is again slightly more general, as we can admit Lipschitz sets. We still indicate by $q_\Omega$ the exponent defined in \eqref{qomega}.
\begin{teo}
\label{teo:B}
Let $\Omega\subset\mathbb{R}^N$ be an open bounded Lipschitz set, with homogeneity index $\alpha_\Omega$. Let us suppose that $\Omega$ has a finite number of connected components. Then, for every $q_\Omega<q<2$,
the first $q-$eigenvalue 
\[
\lambda_1(\Omega;q)=\min_{u\in\mathcal{D}^{1,2}_0(\Omega)} \left\{\int_\Omega |\nabla u|^2\,dx\, :\, \int_\Omega |u|^q\,dx=1\right\},
\]
is isolated in $\mathrm{Spec}(\Omega;q)$.
\end{teo}
Finally, for ease of exposition, we find it useful to state the previous results for $N=2$. Here, the interplay between the cone index $\beta_\Omega$ and the exponent $q_\Omega$ is cleaner. Indeed, by recalling \eqref{cono2}, we have
\[
\alpha_\Omega=\frac{\pi}{2\,\arccos(\beta_\Omega)},
\]
thus we get the following
\begin{coro}[Two dimensional case]
Let $\Omega\subset\mathbb{R}^2$ be an open bounded Lipschitz set with cone index $\beta_\Omega$. The conclusions of Theorems \ref{teo:A} and \ref{teo:B} hold for every 
\[
\max\left\{2-\frac{4}{\pi}\,\arccos(\beta_\Omega),\,1\right\}<q<2.
\]
\end{coro}

\subsection{Non-negative solutions with higher energy}
In the next example we show that when $\Omega$ is not connected, Theorem A {\it can not} be extended to non-negative solutions not having least energy (recall Remark \ref{oss:multiple}). This is similar to the example of \cite[Theorem 3.1]{BFcontro}.
\begin{exa}
\label{exa:next}
Let $1<q<2$ and let $\Omega=\Omega_1\cup \Omega_2$, with $\Omega_1,\Omega_2\subset\mathbb{R}^N$ open bounded sets with smooth boundary, such that $\Omega_1\cap \Omega_2=\emptyset$. Then there exists a sequence $\{U_n\}_{n\in\mathbb{N}}\subset\mathcal{D}^{1,2}_0(\Omega)$ of distinct solutions of the Lane-Emden equation \eqref{LE} and a positive solution $U$ of the same equation, such that 
\[
\lim_{n\to\infty} \|\nabla U_n-\nabla U\|_{L^2(\Omega)}=0.
\]
We start by taking the positive least energy solution $w_{\Omega_1,q}$ of $\Omega_1$. We consider it to be extended by $0$ on the whole $\Omega$. We then take $u_n\in\mathcal{D}^{1,2}_0(\Omega_2)$ to be a $q-$eigenfunction of $\Omega_2$ with unit $L^q$ norm, associated to the $n-$th variational eigenvalue $\lambda_{n,LS}(\Omega_2;q)$ of $\Omega_2$ (recall the definition \eqref{CFWLS}). Again, we consider it to be extended by $0$ on the whole $\Omega$. We then set
\[
U_n=w_{\Omega_1,q}+\lambda_{n,LS}(\Omega_2;q)^{\frac{1}{q-2}}\,u_n,
\]
which solves, by construction
\[
-\Delta U_n=|U_n|^{q-2}\,U_n,\qquad \text{ in }\Omega.
\]
By recalling that $\lambda_{n,LS}(\Omega_2;q)$ diverges to $+\infty$ as $n$ goes to $\infty$ and using that $2-q<2$, we then obtain
\[
\begin{split}
\lim_{n\to\infty} \|\nabla U_n-\nabla w_{\Omega_1,q}\|_{L^2(\Omega)}&=\lim_{n\to\infty} \lambda_{n,LS}(\Omega_2;q)^{\frac{1}{q-2}}\,\|\nabla u_n\|_{L^2(\Omega)}\\
&=\lim_{n\to\infty} \lambda_{n,LS}(\Omega_2;q)^{\frac{1}{q-2}+\frac{1}{2}}=0,
\end{split}
\]
which is the desired conclusion. 
\end{exa}

\appendix 
\section{A pointwise inequality}
The following simple inequality has been useful in order to prove our main result.
\begin{lm}
\label{lm:quasilipschitz}
Let $0<\alpha<1$, then for every $a>0$ and $b\in\mathbb{R}$ we have 
\[
|a^\alpha-|b|^{\alpha-1}\,b|\le 2^{1-\alpha}\,a^{\alpha-1}\,|a-b|.
\]
\end{lm}
\begin{proof}
We first suppose that $a\ge b\ge 0$, then we write (recall that $a>0$)
\[
\begin{split}
|a^\alpha-|b|^{\alpha-1}\,b|=a^\alpha-b^\alpha&=a^\alpha\,\left(1-\left(\frac{b}{a}\right)^\alpha\right)\\
&\le a^\alpha\,\left(1-\frac{b}{a}\right)=a^{\alpha-1}\,(a-b),
\end{split} 
\]
where we used that $t\le t^\alpha$ for every $0\le t\le 1$.
\par
We now suppose that $b>a>0$, then by proceeding as before, we find
\[
|a^\alpha-|b|^{\alpha-1}\,b|=b^\alpha-a^\alpha\le b^{\alpha-1}\,(b-a).
\]
By observing that the power $\alpha-1$ is negative and using the hypothesis $b>a>0$, we prove the inequality in this case, as well.
\par
Finally, we suppose that $a>0\ge b$. In this case, by using the concavity of the map $t\mapsto t^\alpha$, we obtain
\[
|a^\alpha-|b|^{\alpha-1}\,b|=a^\alpha+(-b)^{\alpha}\le 2^{1-\alpha}\,(a-b)^\alpha=2^{1-\alpha}\,(a-b)^{\alpha-1}\,(a-b).
\]
Since $b$ is negative and $\alpha-1<0$, we can further use that $(a-b)^{\alpha-1}\le a^{\alpha-1}$ and get the desired conclusion.
\end{proof}

\section{A uniform $L^\infty$ estimate}

For $1<q<2$, solutions of the Lane-Emden equation enjoys the following {\it universal} global $L^\infty$ estimate. This fact should be well-known, it is mentioned for example in \cite[page 149]{Da1988}. We provide for completeness a precise estimate, under optimal assumptions on the set.
\begin{prop}
\label{prop:universal}
Let $1<q<2$ and let $\Omega\subset\mathbb{R}^N$ be an open set, such that
\[
\lambda_1(\Omega;q)=\min_{u\in\mathcal{D}^{1,2}_0(\Omega)} \left\{\int_\Omega |\nabla u|^2\,dx\, :\, \int_\Omega |u|^q\,dx=1\right\}>0.
\] 
For every solution $u\in \mathcal{D}^{1,2}_0(\Omega)$ of the Lane-Emden equation \eqref{LE}, 
we have $u\in L^\infty(\Omega)$, with the universal estimate
\[
\|u\|_{L^\infty(\Omega)}\le \left\{\begin{array}{ll}
C_{N,q}\,\lambda_1(\Omega;q)^{\frac{q}{2\,(q-2)}\,\frac{2^*-2}{2^*-q}},& \text{ if } N\ge 3,\\
&\\
C_q\,\lambda_1(\Omega;q)^\frac{q}{2\,(q-2)},& \text{ if } N=2.
\end{array}
\right.
\]
\end{prop}
\begin{proof}
By setting $\lambda=\|u\|_{L^q(\Omega)}^{q-2}$, we see that $u$ solves
\[
-\Delta u=\lambda\,\|u\|_{L^q(\Omega)}^{2-q}\,|u|^{q-2}\,u.
\]
We can then apply the estimate of \cite[Proposition 2.5]{BF} and obtain
\[
\|u\|_{L^\infty(\Omega)}\le \left\{\begin{array}{ll}
C_{N,q}\,\Big(\sqrt{\lambda}\Big)^\frac{2^*}{2^*-q}\,\|u\|_{L^q(\Omega)},& \text{ if } N\ge 3,\\
&\\
C_q\,\sqrt{\dfrac{\lambda}{\lambda_1(\Omega;q)}}\,\sqrt{\lambda}\,\|u\|_{L^q(\Omega)},& \text{ if } N=2.
\end{array}
\right.
\]
By recalling the definition of $\lambda$, we obtain the $L^\infty-L^q$ estimate
\[
\|u\|_{L^\infty(\Omega)}\le \left\{\begin{array}{ll}
C_{N,q}\,\|u\|_{L^q(\Omega)}^{\frac{q}{2}\,\frac{2^*-2}{2^*-q}},& \text{ if } N\ge 3,\\
&\\
C_q\,\sqrt{\dfrac{1}{\lambda_1(\Omega;q)}}\,\|u\|^{q-1}_{L^q(\Omega)},& \text{ if } N=2.
\end{array}
\right.
\]
We only need to show that the $L^q$ norm admits a universal estimate. For this, from the equation we have the energy identity
\[
\int_\Omega |\nabla u|^2\,dx=\int_\Omega |u|^q\,dx.
\]
By using the definition of $\lambda_1(\Omega;q)$, this entails that
\[
\lambda_1(\Omega;q)\,\left(\int_\Omega |u|^q\,dx\right)^\frac{2}{q}\le \int_\Omega |u|^q\,dx.
\]
By using that $2/q>1$, we obtain the desired conclusion.
\end{proof}
\begin{oss}
The previous estimate guarantees that the $L^\infty$ norm of a solution of the Lane-Emden equation can be controlled from above in terms of a (negative) power of the sharp Poincar\'e-Sobolev constant $\lambda_1(\Omega;q)$. This estimate can not be reversed. Indeed, by taking the ``slab--type'' sequence
\[
\Omega_n=(-n,n)^{N-1}\times(-1,1),
\]
we know that the positive least energy solution $w_{\Omega_n,q}$ can be bounded uniformly in $L^\infty(\Omega)$, see \cite[Proposition 4.3]{BFR}. On the other hand, we have
\[
\lim_{n\to\infty} \lambda_1(\Omega_n;q)=0.
\]
\end{oss}

The previous result permits to infer that on the space of the solutions of the Lane-Emden equation in $\Omega$, the $L^1(\Omega)$ strong topology and the $\mathcal{D}^{1,2}_0(\Omega)$ strong topology are actually equivalent.
\begin{coro}
\label{coro:accesso}
Let $1<q<2$ and let $\Omega\subset\mathbb{R}^N$ be an open set with finite measure.
There exists a constant $C>0$ depending on $N,q$ and $\lambda_1(\Omega;q)$ only, such that for every pair $u,v$ of solutions of the Lane-Emden equation \eqref{LE}, we have
\[
\|\nabla u-\nabla v\|_{L^2(\Omega)}\le C\, \sqrt{\|u-v\|_{L^1(\Omega)}}.
\]
\end{coro}
\begin{proof}
By subtracting the equations satisfied by $u$ and $v$, we get
\[
\int_\Omega \langle \nabla(u-v),\nabla \varphi\rangle\,dx=\int_\Omega (|u|^{q-2}\,u-|v|^{q-2}\,v)\,\varphi\,dx,
\]
for every $\varphi\in \mathcal{D}^{1,2}_0(\Omega)$. 
We use this identity with $\varphi=u-v$, so to get
\[
\begin{split}
\int_\Omega |\nabla u-\nabla v|^2\,dx&=\int_\Omega (|u|^{q-2}\,u-|v|^{q-2}\,v)\,(u-v)\,dx\\
&\le \int_\Omega (|u|^{q-1}+|v|^{q-1})\,|u-v|\,dx\\
&\le \left(\|u\|_{L^\infty(\Omega)}^{q-1}+\|v\|_{L^\infty(\Omega)}^{q-1}\right)\,\|u-v\|_{L^1(\Omega)}.
\end{split}
\]
If we now use the uniform $L^\infty$ estimate of Proposition \ref{prop:universal}, we get the desired conclusion.
\end{proof}

\section{Defect of compactness in convex cones}

In this section, we show that the embedding 
\[
\mathcal{D}^{1,2}_0(\Omega)\hookrightarrow L^2(\Omega;w_{\Omega,q}^{q-2}),
\] 
fails to be compact in a narrow convex cone. For completeness, we will make a more refined analysis, aiming at identifying the energy levels at which the loss of compactness occurs. We will see that this is linked to the exact determination of a Hardy-type sharp constant.
\vskip.2cm\noindent
Throughout this section, we still use the notation of Definition \ref{defi:cones} and still set
\[
\Phi(t)=t\,(N-2+t),\qquad \text{ for }t\ge 0.
\] 
Recall that this is a monotone increasing function. Then the main outcome of this appendix will be the following
\begin{prop}
\label{prop:trac}
Let $1<q<2$ and let $0\le \beta<1$ be such that
\[
\Phi\left(\frac{2}{2-q}\right)<\lambda(\mathcal{S}(\beta)).
\]
For every $0<R<+\infty$, we define the ``concentration energy at the tip'' of the convex cone $\Gamma(\beta,R)$ as the quantity
\begin{equation}
\label{tips}
\inf_{\varphi\in C^\infty_0(\Gamma(\beta,R))} \left[\lim_{n\to\infty} \frac{\displaystyle \int_{\Gamma(\beta,R)}|\nabla \varphi_n|^2\,dx}{\displaystyle\int_{\Gamma(\beta,R)}|\varphi_n|^2\,w_{\Gamma(\beta,R),q}^{q-2}\,dx}\right],
\end{equation}
where for every $\varphi\in C^\infty_0(\Gamma(\beta,R))$, we set
\[
\varphi_n(x)=n^\frac{N-2}{2}\,\varphi(n\,x),\qquad n\in\mathbb{N}.
\]
Then we have: 
\begin{itemize}
\item[1)] the value \eqref{tips} does not depend on $R$, we indicate it by $\mathcal{C}_q(\beta)$. Moreover, we have $\mathcal{C}_q(\beta)>1$ and 
\[
\lim_{\beta\to 1^-} \mathcal{C}_q(\beta)=1;
\]
\vskip.2cm
\item[2)] for every $t>\mathcal{C}_q(\beta)$, the set
\[
\mathcal{E}_{\Gamma(\beta,R),q}(t)=\left\{\varphi\in \mathcal{D}^{1,2}_0(\Gamma(\beta,R))\, :\, \int_{\Gamma(\beta,R)} w_{\Gamma(\beta,R),q}^{q-2}\,|\varphi|^2\,dx=1,\ \int_{\Gamma(\beta,R)} |\nabla \varphi|^2\,dx\le t\right\},
\]
is not precompact in $L^2(\Gamma(\beta,R);w_{\Gamma(\beta,R),q}^{q-2})$.
\end{itemize}
\end{prop}
Before giving the proof of Proposition \ref{prop:trac}, we need some intermediate expedient results.

\begin{lm}[A special solution]
\label{lm:coni}
Let $1<q<2$ and let $\beta$ be such that
\begin{equation}
\label{beta}
\Phi\left(\frac{2}{2-q}\right)<\lambda(\mathcal{S}(\beta)).
\end{equation}
Then there exists a positive function $\psi\in \mathcal{D}^{1,2}_0(\mathcal{S}(\beta))\cap L^\infty(\mathcal{S}(\beta))$ such that 
\[
V(x)=|x|^\frac{2}{2-q}\,\psi\left(\frac{x}{|x|}\right),
\]
is a positive solution of 
\[
-\Delta V=V^{q-1},\ \text{ in } \Gamma(\beta,+\infty),\qquad V=0,\ \text{ on }\partial\Gamma(\beta,+\infty).
\]
Moreover, for every $R>0$ we have
\begin{equation}
\label{stiacciata}
w_{\Gamma(\beta,R),q}(x)\le V(x),\qquad \text{ for } x\in \Gamma(\beta,R).
\end{equation}
Finally, we have 
\begin{equation}
\label{dontmess}
\lim_{R\to +\infty}w_{\Gamma(\beta,R),q}=V,\qquad \mbox{ uniformly on every } \Gamma(\beta,r).
\end{equation}
\end{lm}
\begin{proof}
We start by considering the variational problem
\[
\mu_q(\beta)=\min_{\varphi\in \mathcal{D}^{1,2}_0(\mathcal{S}(\beta))\setminus\{0\}} \frac{\displaystyle\int_{\mathcal{S}(\beta)} |\nabla_\tau\varphi|^2\,d\mathcal{H}^{N-1}-\Phi\left(\frac{2}{2-q}\right)\,\int_{\mathcal{S}(\beta)} |\varphi|^2\,d\mathcal{H}^{N-1}}{\displaystyle\left(\int_{\mathcal{S}(\beta)} |\varphi|^q\,d\mathcal{H}^{N-1}\right)^\frac{2}{q}},
\]
where $\nabla_\tau$ denotes the tangential gradient.
By definition of $\lambda(\mathcal{S}(\beta))$, we have
\[
\int_{\mathcal{S}(\beta)} |\nabla_\tau \varphi|^2\,d\mathcal{H}^{N-1}\ge \lambda(\mathcal{S}(\beta))\,\int_{\mathcal{S}(\beta)} |\varphi|^2\,d\mathcal{H}^{N-1},\qquad \text{ for }\varphi\in \mathcal{D}^{1,2}_0(\mathcal{S}(\beta)).
\]
Then, keeping in mind the choice \eqref{beta} of $\beta$, by applying the Direct Methods in the Calculus of Variations
we easily get that the value $\mu_q(\beta)$ is attained by a function $\widetilde\psi$, which can be taken to be positive and normalized by the condition
\begin{equation}
\label{normalepsi}
\int_{\mathcal{S}(\beta)} |\widetilde\psi|^q\,d\mathcal{H}^{N-1}=1.
\end{equation}
Moreover, still thanks to \eqref{beta}, we can assure that $\mu_q(\beta)>0$. We now observe that $\widetilde\psi$ weakly solves
\[
-\Delta_g \widetilde\psi-\Phi\left(\frac{2}{2-q}\right)\,\widetilde\psi=\mu_q(\beta)\,\widetilde\psi^{q-1},\qquad \text{ in } \mathcal{S}(\beta),
\]
where $\Delta_g$ is the Laplace-Beltrami operator on $\mathbb{S}^{N-1}$.
If we now set 
\begin{equation}
\label{scegli}
\psi=\mu_q(\beta)^{-\frac{1}{2-q}} \widetilde\psi,
\end{equation}
this function solves
\begin{equation}
\label{angular}
-\Delta_g \psi-\Phi\left(\frac{2}{2-q}\right)\,\psi=\psi^{q-1}.
\end{equation}
By writing the Laplacian in spherical coordinates, it is easily seen that the function
\[
V(x)=|x|^\frac{2}{2-q}\,\psi\left(\frac{x}{|x|}\right),
\]
has the claimed properties. 
\vskip.2cm\noindent
We now prove the property \eqref{stiacciata}. For this, we use a comparison principle similar to that of \cite[Lemma 2.7]{BFR}.
We observe that the restriction of $V$ to $\Gamma(\beta,R)$ is the unique solution of the variational problem
\begin{equation}
\label{troncata}
\min_{\varphi\in W^{1,2}(\Gamma(\beta,R))} \left\{\frac{1}{2}\,\int_{\Gamma(\beta,R)} |\nabla \varphi|^2\,dx-\frac{1}{q}\,\int_{\Gamma(\beta,R)}\varphi^q\,dx\, :\, \varphi\ge 0,\ \varphi=V \text{ on } \partial\Gamma(\beta,R)\right\}.
\end{equation}
We test the minimality of $V$ in \eqref{troncata} comparing with the value corresponding to the function
\(
\varphi = \max\{w_{\Gamma(\beta,R),q},\,V\}.
\)
In this way, we get
\begin{equation}
\label{UV}
\begin{split}
\frac{1}{2}\,\int_{\{w_{\Gamma(\beta,R),q}>V\}} |\nabla w_{\Gamma(\beta,R),q}|^2\,dx&-\frac{1}{q}\,\int_{\{w_{\Gamma(\beta,R),q}>V\}}w_{\Gamma(\beta,R),q}^q\,dx\\
&\ge \frac{1}{2}\,\int_{\{w_{\Gamma(\beta,R),q}>V\}} |\nabla V|^2\,dx-\frac{1}{q}\,\int_{\{w_{\Gamma(\beta,R),q}>V\}}V^q\,dx.
\end{split}
\end{equation}
If we now introduce 
\[
\widetilde\varphi=\min\big\{w_{\Gamma(\beta,R),q},V\big\},
\] 
and add on both sides of \eqref{UV} the term 
\[
\frac{1}{2}\,\int_{\{w_{\Gamma(\beta,R),q}\le V\}} |\nabla w_{\Gamma(\beta,R),q}|^2\,dx-\frac{1}{q}\,\int_{\{w_{\Gamma(\beta,R),q}\le V\}}w_{\Gamma(\beta,R),q}^q\,dx,
\]
we get 
\begin{equation}
\label{unaltro}
\frac{1}{2}\,\int_{\Gamma(\beta,R)} |\nabla w_{\Gamma(\beta,R),q}|^2\,dx-\frac{1}{q}\,\int_{\Gamma(\beta,R)}w_{\Gamma(\beta,R),q}^q\,dx\ge \frac{1}{2}\,\int_{\Gamma(\beta,R)} |\nabla \widetilde\varphi|^2\,dx-\frac{1}{q}\,\int_{\Gamma(\beta,R)}\widetilde \varphi^q\,dx.
\end{equation}
On the other hand, by Definition \ref{defi:boundeddensity} the function $w_{\Gamma(\beta,R),q}$ is the unique solution of 
\[
\min_{\varphi\in \mathcal{D}^{1,2}_0(\Gamma(\beta,R))} \left\{\frac{1}{2}\,\int_{\Gamma(\beta,R)} |\nabla \varphi|^2\,dx-\frac{1}{q}\,\int_{\Gamma(\beta,R)}\varphi^q\,dx\, :\, \varphi\ge 0\right\}.
\]
Hence, since $\widetilde\varphi$ is admissible for this problem, equation \eqref{unaltro} shows that 
\[
\widetilde\varphi=\min\big\{w_{\Gamma(\beta,R),q},V\big\}=w_{\Gamma(\beta,R),q},
\]
which is the desired estimate \eqref{stiacciata}.
\vskip.2cm\noindent
Finally, we prove \eqref{dontmess}. We first observe that by  \cite[Lemma 2.7]{BFR}, we get that 
\[
w_{\Gamma(\beta,R_1),q}\le w_{\Gamma(\beta,R_2),q},\qquad \mbox{ for every } R_1\le R_2.
\]
Thus the pointwise limit  
\[
W(x):=\lim_{R\to+\infty} w_{\Gamma(\beta,R),q}(x),
\]
exists by monotonicity and it is finite, thanks to \eqref{stiacciata}. Moreover, by proceeding as in the proof of \cite[Proposition 5.1]{BFR}, it is not difficult to see that $W$ solves the Lane-Emden equation in the infinite cone $\Gamma(\beta,+\infty)$. By using the scaling properties of the Lane-Emden equation and the uniqueness of the positive least energy solution in $\Gamma(\beta,R)$, we observe that
\begin{equation}
\label{scalingw}
w_{\Gamma(\beta,R),q}(x)=R^\frac{2}{2-q}\,w_{\Gamma(\beta,1),q}\left(\frac{x}{R}\right),\qquad \mbox{ for }x\in \Gamma(\beta,R).
\end{equation}
Thus, for every $\lambda>0$, we have
\[
\begin{split}
W(\lambda\,x)=\lim_{R\to+\infty} w_{\Gamma(\beta,R),q}(\lambda\,x)&=\lim_{R\to+\infty}R^\frac{2}{2-q}\,w_{\Gamma(\beta,1),q}\left(\frac{\lambda\,x}{R}\right)\\
&=\lambda^\frac{2}{2-q}\,\lim_{R\to+\infty}\left(\frac{R}{\lambda}\right)^\frac{2}{2-q}\,w_{\Gamma(\beta,1),q}\left(\frac{\lambda\,x}{R}\right)\\
&=\lambda^\frac{2}{2-q}\,\lim_{R\to +\infty} w_{\Gamma(\beta,R/\lambda)}(x)\\
&=\lambda^\frac{2}{2-q}\,W(x).
\end{split}
\]
This shows that $W$ is $2/(2-q)-$homogeneous, so that it can be written as 
\[
W(x)=|x|^\frac{2}{2-q}\,W\left(\frac{x}{|x|}\right),\qquad \mbox{ for every } x\in \Gamma(\beta,+\infty).
\]
In order to conclude, we just need to show that
\[
\psi(\omega)=W(\omega),\qquad \mbox{ for every } \omega\in\mathcal{S}(\beta),
\]
where $\psi$ is still the function defined in \eqref{scegli}.
\par
Since $W$ solves the Lane-Emden equation, by writing the Laplacian in spherical coordinates we get that $\omega\to W(\omega)$ must be a positive solution of equation \eqref{angular}, as well. We now adapt the trick by Brezis and Oswald (see \cite{BO} and also \cite[Lemma 2.2]{BFR}), based on Picone's inequality, in order to show uniqueness for \eqref{angular}. We take the weak formulations
\[
\int_{\mathcal{S}(\beta)} \langle \nabla_\tau\psi,\nabla_\tau \varphi\rangle\,d\mathcal{H}^{N-1}-\Phi\left(\frac{2}{2-q}\right)\,\int_{\mathcal{S}(\beta)}\psi\,\varphi\,d\mathcal{H}^{N-1}=\int_{\mathcal{S}(\beta)} \psi^{q-1}\,\varphi\,d\mathcal{H}^{N-1},
\]
and
\[
\int_{\mathcal{S}(\beta)} \langle \nabla_\tau W,\nabla_\tau \varphi\rangle\,d\mathcal{H}^{N-1}-\Phi\left(\frac{2}{2-q}\right)\,\int_{\mathcal{S}(\beta)}W\,\varphi\,d\mathcal{H}^{N-1}=\int_{\mathcal{S}(\beta)} W^{q-1}\,\varphi\,d\mathcal{H}^{N-1}.
\]
Then, for every $\varepsilon>0$, we insert the test function $\varphi=(W^2/(\varepsilon+\psi)-\psi)$ in the first equation and the test function $\varphi=(\psi^2/(\varepsilon+W)-W)$ in the second one. By summing up the resulting identities, we get
\[
\begin{split}
\int_{\mathcal{S}(\beta)}& \left\langle \nabla_\tau\psi,\nabla_\tau \left(\frac{W^2}{\varepsilon+\psi}-\psi\right)\right\rangle\,d\mathcal{H}^{N-1}+\int_{\mathcal{S}(\beta)} \left\langle \nabla_\tau W,\nabla_\tau \left(\frac{\psi^2}{\varepsilon+W}-W\right)\right\rangle\,d\mathcal{H}^{N-1}\\
&-\Phi\left(\frac{2}{2-q}\right)\,\int_{\mathcal{S}(\beta)}\psi\,\left(\frac{W^2}{\varepsilon+\psi}-\psi\right)\,d\mathcal{H}^{N-1}-\Phi\left(\frac{2}{2-q}\right)\,\int_{\mathcal{S}(\beta)}W\,\left(\frac{\psi^2}{\varepsilon+W}-W\right)\,d\mathcal{H}^{N-1}\\
&=\int_{\mathcal{S}(\beta)} \psi^{q-1}\,\left(\frac{W^2}{\varepsilon+\psi}-\psi\right)\,d\mathcal{H}^{N-1}+\int_{\mathcal{S}(\beta)} W^{q-1}\,\left(\frac{\psi^2}{\varepsilon+W}-W\right)\,d\mathcal{H}^{N-1}.
\end{split}
\]
We now use {\it Picone's inequality}, so that 
\[
\int_{\mathcal{S}(\beta)} \left\langle \nabla_\tau\psi,\nabla_\tau \frac{W^2}{\varepsilon+\psi}\right\rangle\,d\mathcal{H}^{N-1}=\int_{\mathcal{S}(\beta)} \left\langle \nabla_\tau(\psi+\varepsilon),\nabla_\tau \frac{W^2}{\varepsilon+\psi}\right\rangle\,d\mathcal{H}^{N-1}\le \int_{\mathcal{S}(\beta)} |\nabla_\tau W|^2\,d\mathcal{H}^{N-1},
\]
and 
\[
\int_{\mathcal{S}(\beta)} \left\langle \nabla_\tau W,\nabla_\tau \frac{\psi^2}{\varepsilon+W}\right\rangle\,d\mathcal{H}^{N-1}=\int_{\mathcal{S}(\beta)} \left\langle \nabla_\tau(W+\varepsilon),\nabla_\tau \frac{\psi^2}{\varepsilon+W}\right\rangle\,d\mathcal{H}^{N-1}\le \int_{\mathcal{S}(\beta)} |\nabla_\tau \psi|^2\,d\mathcal{H}^{N-1}.
\]
A further passage to the limit as $\varepsilon$ goes to $0$ leads to
\[
\int_{\mathcal{S}(\beta)} \left(\psi^{q-2}\,W^2-\psi^q\right)\,d\mathcal{H}^{N-1}+\int_{\mathcal{S}(\beta)} \left(W^{q-2}\,\psi^2 -W^q\right)\,d\mathcal{H}^{N-1}\le 0.
\]
The last two integrals can be rearranged as follows
\[
\int_{\mathcal{S}(\beta)} (\psi^{q-2}-W^{q-2})\,(W^2-\psi^2)\,d\mathcal{H}^{N-1}\le 0.
\]
On the other hand, by virtue of the fact that $q<2$, we have
\[
(a^{q-2}-b^{q-2})\,(b^2-a^2)>0,\qquad \text{ for every } a,b>0 \text{ such that } a\not=b.
\]
The last two displays shows that we must have $\psi=W$ on $\mathcal{S}(\beta)$. The proof is now complete.
\end{proof}
\begin{oss}
\label{oss:inversa}
We observe that from the equation \eqref{angular}, we have
\[
\begin{split}
\int_{\mathcal{S}(\beta)}\psi^q\,d\mathcal{H}^{N-1}&=\int_{\mathcal{S}(\beta)}|\nabla_\tau\psi|^2\,d\mathcal{H}^{N-1}-\Phi\left(\frac{2}{2-q}\right)\,\int_{\mathcal{S}(\beta)}\psi^2\,d\mathcal{H}^{N-1}\\
&\ge \left(\lambda(\mathcal{S}(\beta))-\Phi\left(\frac{2}{2-q}\right)\right)\,\int_{\mathcal{S}(\beta)} \psi^2\,d\mathcal{H}^{N-1},
\end{split}
\]
where we also used Poincar\'e's inequality on $\mathcal{S}(\beta)$. By recalling \eqref{normalepsi} and \eqref{scegli}, we also have
\[
\int_{\mathcal{S}(\beta)}\psi^q\,d\mathcal{H}^{N-1}=\mu_q(\beta)^{-\frac{1}{2-q}}.
\] 
\end{oss}

\begin{lm}
\label{lm:1}
Let $1<q<2$ and let $0\le \beta<1$ be such that
\[
\Phi\left(\frac{2}{2-q}\right)<\lambda(\mathcal{S}(\beta)).
\]
For every $0<R\le +\infty$, we define the Hardy-type constant
\[
\sigma(\beta,R)=\inf_{\varphi\in C^\infty_0(\Gamma(\beta,R))\setminus\{0\}} \frac{\displaystyle \int_{\Gamma(\beta,R)}|\nabla \varphi|^2\,dx}{\displaystyle\int_{\Gamma(\beta,R)} |\varphi|^2\,V^{q-2}\,dx},
\]
where $V$ is the function of Lemma \ref{lm:coni}.
Then we have
\[
\sigma(\beta,R)=\sigma(\beta,+\infty)=:\sigma(\beta)> 1,
\]
and
\[
\lim_{\beta\to 1^-}\sigma(\beta)=1.
\]
\end{lm}
\begin{proof}
We first prove that $\sigma(\beta,R)=\sigma(\beta,+\infty)$, then we show that $\sigma(\beta,+\infty)> 1$.
\vskip.2cm\noindent
Since $C^\infty_0(\Gamma(\beta,R))\subset C^\infty_0(\Gamma(\beta,+\infty))$, we immediately have that
\[
\sigma(\beta,R)\ge \sigma(\beta,+\infty).
\]
In order to prove the reverse inequality, for every $\varepsilon>0$ we take $\varphi_\varepsilon\in C^\infty_0(\Gamma(\beta,+\infty))$ such that
\[
\frac{\displaystyle \int_{\Gamma(\beta,+\infty)}|\nabla \varphi_\varepsilon|^2\,dx}{\displaystyle\int_{\Gamma(\beta,+\infty)} |\varphi_\varepsilon|^2\,V^{q-2}\,dx}< \sigma(\beta,+\infty)+\varepsilon.
\]
We now take the rescaled function
\[
\varphi_{\varepsilon,n}(x)=n^\frac{N-2}{2}\,\varphi_\varepsilon(n\,x),\qquad \text{ for }n\in\mathbb{N},
\]
and observe that 
\[
\frac{\displaystyle \int_{\Gamma(\beta,+\infty)}|\nabla \varphi_{\varepsilon,n}|^2\,dx}{\displaystyle\int_{\Gamma(\beta,+\infty)} |\varphi_{\varepsilon,n}|^2\,V^{q-2}\,dx}=\frac{\displaystyle \int_{\Gamma(\beta,+\infty)}|\nabla \varphi_\varepsilon|^2\,dx}{\displaystyle\int_{\Gamma(\beta,+\infty)} |\varphi_\varepsilon|^2\,V^{q-2}\,dx},
\]
thanks to the $2-$homogeneity of $V^{2-q}$.
Moreover, for $n$ large enough we also have $\varphi_{\varepsilon,n}\in C^\infty_0(\Gamma(\beta,R))$. This in turn permits to infer that
\[
\sigma(\beta,R)< \sigma(\beta,+\infty)+\varepsilon.
\]
By arbitrariness of $\varepsilon>0$, we obtain that $\sigma(\beta,R)=\sigma(\beta,+\infty)$.
\vskip.2cm\noindent
We are now left with estimating the Hardy-type constant $\sigma(\beta):=\sigma(\beta,+\infty)$. We first prove that $\sigma(\beta)>1$. For this, we recall that the function $V$ satisfies
\[
\int_{\Gamma(\beta,+\infty)} V^{q-1}\,\varphi\,dx=\int_{\Gamma(\beta,+\infty)} \langle \nabla V,\nabla \varphi\rangle\,dx,\qquad \text{ for every } \varphi \in C^\infty_0(\Gamma(\beta,+\infty)).
\]
Given $\eta\in C^\infty_0(\Gamma(\beta,+\infty))$, we use the previous identity with the choice $\varphi=\eta^2/V$.
An application of {\it Picone's identity} leads to
\[
\begin{split}
\int_{\Gamma(\beta,+\infty)} \eta^2\,V^{q-2}\,dx&= \int_{\Gamma(\beta,+\infty)}|\nabla \eta|^2\,dx-\int_{\Gamma(\beta,+\infty)} \left|\eta\,\frac{\nabla V}{V}-\nabla \eta\right|^2\,dx.
\end{split}
\]
By dividing both sides by the weighted $L^2$ norm of $\eta$, we get
\[
\begin{split}
\sigma(\beta)&=\inf_{\eta\in C^\infty_0(\Gamma(\beta,R))\setminus\{0\}} \frac{\displaystyle \int_{\Gamma(\beta,R)}|\nabla \eta|^2\,dx}{\displaystyle\int_{\Gamma(\beta,R)} |\eta|^2\,V^{q-2}\,dx}=1+\inf_{\eta\in C^\infty_0(\Gamma(\beta,R))\setminus\{0\}}\frac{\displaystyle \int_{\Gamma(\beta,R)} \left|\eta\,\frac{\nabla V}{V}-\nabla \eta\right|^2\,dx}{\displaystyle\int_{\Gamma(\beta,R)} |\eta|^2\,V^{q-2}\,dx}.
\end{split}
\]
We perform the change of variable $\eta=\varphi\,V$, so that the last minimization problem can be transformed into
\[
\Theta(\beta):=\inf_{\varphi\in C^\infty_0(\Gamma(\beta,R))\setminus\{0\}}\frac{\displaystyle \int_{\Gamma(\beta,R)} |\nabla \varphi|^2\,V^2\,dx}{\displaystyle\int_{\Gamma(\beta,R)} |\varphi|^2\,V^q\,dx}.
\]
In order to show that $\sigma(\beta)>1$, it is sufficient to prove that $\Theta(\beta)>0$. For this, we use spherical coordinates, the specific form of $V$ and the one-dimensional Hardy's inequality (see \cite[equation (1.3.1)]{Ma}), i.e. 
\[
\int_0^R |f'(t)|^2\,t^{N-1+\frac{4}{2-q}}\,dt\ge C_{N,q}\,\int_0^R |f(t)|^2\,t^{N-1+\frac{2\,q}{2-q}}\,dt.
\]
This is valid for every smooth function $f$, such that $f(R)=0$. We denote by $C_{N,q}>0$ the sharp constant, whose precise value has no bearing in what follows. By proceeding as exposed above, we get
\begin{equation}
\label{lowerboundk}
\begin{split}
\Theta(\beta)&= \inf_{\varphi\in C^\infty_0(\Gamma(\beta,R))\setminus\{0\}}\frac{\displaystyle \int_{\Gamma(\beta,R)} |\nabla \varphi|^2\,|x|^\frac{4}{2-q}\,\psi^2\,dx}{\displaystyle\int_{\Gamma(\beta,R)} |\varphi|^2\,|x|^\frac{2\,q}{2-q}\,\psi^q\,dx}\\
&=\inf_{\varphi\in C^\infty_0(\Gamma(\beta,R))\setminus\{0\}}\frac{\displaystyle \int_{\mathcal{S}(\beta)}\left(\int_{0}^R \Big[|\partial_\varrho\varphi|^2+\varrho^{-2}\,|\nabla_\tau \varphi|^2\Big]\,\varrho^{N-1+\frac{4}{2-q}}\,d\varrho\right)\,\psi^2\,dx}{\displaystyle\int_0^R\left(\int_{\mathcal{S}(\beta)}|\varphi|^2\,\psi^q\,d\mathcal{H}^{N-1}\right)\,\varrho^{N-1+\frac{2\,q}{2-q}}\,d\varrho}\\
&\ge\inf_{\varphi\in C^\infty_0(\Gamma(\beta,R))\setminus\{0\}}\frac{\displaystyle\int_0^R \left(\int_{\mathcal{S}(\beta)}\Big[ C_{N,q}\,|\varphi|^2+|\nabla_\tau \varphi|^2\Big]\,\psi^2\,d\mathcal{H}^{N-1}\right)\,\varrho^{N-1+\frac{2\,q}{2-q}}\,d\varrho}{\displaystyle \int_0^R \left(\int_{\mathcal{S}(\beta)}|\varphi|^2\,\psi^q\,d\mathcal{H}^{N-1}\right)\,\varrho^{N-1+\frac{2\,q}{2-q}}\,d\varrho}.
\end{split}
\end{equation}
We now observe that there exists a constant $C>0$ such that
\begin{equation}
\label{doublesided}
\frac{1}{C}\,\mathrm{dist}_g(\omega,\partial\mathcal{S}(\beta))\le\psi(\omega)\le C\,\mathrm{dist}_g(\omega,\partial\mathcal{S}(\beta)),\qquad \text{ for }\omega\in\mathcal{S}(\beta).
\end{equation}
Here we denote by $\mathrm{dist}_g(\cdot,\partial\mathcal{S}(\beta))$ the geodesic distance on $\mathcal{S}(\beta)$ from the boundary. Thus, for every $\varrho\in[0,R]$, we can apply the weighted Poincar\'e inequality of \cite[Theorem 8.2]{Ku} to the compactly supported function $\omega\mapsto \varphi(\varrho\,\omega)$. This gives
\[
\begin{split}
\int_{\mathcal{S}(\beta)}\Big[ C_{N,q}\,|\varphi|^2+|\nabla_\tau \varphi|^2\Big]\,\psi^2\,d\mathcal{H}^{N-1}&\ge \gamma\,\int_{\mathcal{S}(\beta)} |\varphi|^2\,d\mathcal{H}^{N-1}\\
&\ge \frac{\gamma}{\left(\max\limits_{\omega\in\mathcal{S}(\beta)}\mathrm{dist}_g(\omega,\partial\mathcal{S}(\beta))\right)^q}\,\int_{\mathcal{S}(\beta)} |\varphi|^2\,\psi^q\,d\mathcal{H}^{N-1},
\end{split}
\]
for a suitable constant $\gamma>0$. By inserting this estimate in \eqref{lowerboundk}, we get $\Theta(\beta)>0$ as desired.
\vskip.2cm\noindent
Finally, we show that $\sigma(\beta)\to 1$ as $\beta$ goes to $1$, i.e.
\begin{equation}
\label{theta1}
\lim_{\beta\to 1^-} \Theta(\beta)=0.
\end{equation}
For every $\varepsilon>0$, by definition of sharp constant we know that there exists $f_\varepsilon$ such that  
\begin{equation}
\label{almostsharp}
\int_0^R |f'_\varepsilon(t)|^2\,t^{N-1+\frac{4}{2-q}}\,dt< C_{N,q}\,(1+\varepsilon)\,\int_0^R |f_\varepsilon(t)|^2\,t^{N-1+\frac{2\,q}{2-q}}\,dt.
\end{equation}
We then take $g\in C^\infty_0(\mathcal{S}(\beta))$ and insert the test function $\varphi(x)=f_\varepsilon(|x|)\,g(x/|x|)$ in the minimization problem which defines $\Theta(\beta)$. By using spherical coordinates, recalling the definition of $V$ and using \eqref{almostsharp}, we get
\[
\begin{split}
\Theta(\beta)
&\le C_{N,q}\,(1+\varepsilon)\,\frac{\displaystyle \int_{\mathcal{S}(\beta)}|g|^2\,\psi^2\,d\mathcal{H}^{N-1}}{\displaystyle  \int_{\mathcal{S}(\beta)}|g|^2\,\psi^q\,d\mathcal{H}^{N-1}}+\frac{\displaystyle \int_{\mathcal{S}(\beta)}|\nabla_\tau g|^2\,\psi^2\,d\mathcal{H}^{N-1}}{\displaystyle  \int_{\mathcal{S}(\beta)}|g|^2\,\psi^q\,d\mathcal{H}^{N-1}}.
\end{split}
\]
By taking the limit as $\varepsilon$ goes to $0$, this gives
\[
\Theta(\beta)\le \frac{\displaystyle \int_{\mathcal{S}(\beta)}\left[C_{N,q}\,|g|^2+|\nabla_\tau g|^2\right]\,\psi^2\,d\mathcal{H}^{N-1}}{\displaystyle  \int_{\mathcal{S}(\beta)}|g|^2\,\psi^q\,d\mathcal{H}^{N-1}},
\]
for every $g\in C^\infty_0(\mathcal{S}(\beta))$. In particular, for every $\varepsilon>0$ we take a compactly supported approximation of the unit $g_\varepsilon$, i.e. $g_\varepsilon\in C^\infty_0(\mathcal{S}(\beta))$ with
\[
0\le g_\varepsilon\le 1,\qquad g_\varepsilon\equiv 1 \text{ on } \mathcal{S}(\beta+\varepsilon),\qquad |\nabla_\tau g_\varepsilon|\le \frac{C}{\varepsilon}.
\]
We also use that
\[
|\psi|\le C'\,\varepsilon,\qquad \text{ on } \mathcal{S}(\beta)\setminus\mathcal{S}(\beta+\varepsilon),
\]
which follows from \eqref{doublesided}.
Then we obtain
\[
\Theta(\beta)\le \frac{\displaystyle C_{N,q}\,\int_{\mathcal{S}(\beta)}|g_\varepsilon|^2\,\psi^2\,d\mathcal{H}^{N-1}+(C')^2\,C\,\mathcal{H}^{N-1}(\mathcal{S}(\beta)\setminus\mathcal{S}(\beta+\varepsilon))}{\displaystyle  \int_{\mathcal{S}(\beta)}|g_\varepsilon|^2\,\psi^q\,d\mathcal{H}^{N-1}}.
\]
By taking the limit as $\varepsilon$ goes to $0$, we obtain the estimate
\[
\Theta(\beta)\le C_{N,q}\,\frac{\displaystyle \int_{\mathcal{S}(\beta)}\psi^2\,d\mathcal{H}^{N-1}}{\displaystyle \int_{\mathcal{S}(\beta)}\psi^q\,d\mathcal{H}^{N-1}}.
\]
In particular, by recalling Remark \ref{oss:inversa}, we end up with the upper bound
\[
\Theta(\beta)\le C_{N,q}\,\left(\lambda(\mathcal{S}(\beta))-\Phi\left(\frac{2}{2-q}\right)\right)^{-1}.
\]
Thus the claimed asymptotic behavior \eqref{theta1} of $\Theta(\beta)$ is proved. 
\end{proof}
We can now prove the main result of this section.
\begin{proof}[Proof of Proposition \ref{prop:trac}]
Let us fix a nontrivial function $\varphi\in C^\infty_0(\Gamma(\beta,R))$ and take the sequence
\[
\varphi_n(x)=n^\frac{N-2}{2}\,\varphi(n\,x),\qquad n\in\mathbb{N}.
\]
It is not difficult to see that $\{\varphi_n\}_{n\in\mathbb{N}}$ is bounded in $\mathcal{D}^{1,2}_0(\Gamma(\beta,R))$ and it converges to $0$, strongly in $L^2(\Gamma(\beta,R))$. Indeed, we have
\begin{equation}
\label{hardy1}
\int_{\Gamma(\beta,R)} |\nabla \varphi_n|^2\,dx=\int_{\Gamma(\beta,R/n)} |\nabla \varphi_n|^2\,dx=\int_{\Gamma(\beta,R)} |\nabla \varphi|^2\,dx,
\end{equation}
and
\begin{equation}
\label{hardy2}
\int_{\Gamma(\beta,R)} |\varphi_n|^2\,dx=\int_{\Gamma(\beta,R/n)} |\varphi_n|^2\,dx=n^{-2}\,\int_{\Gamma(\beta,R)} |\varphi|^2\,dx.
\end{equation}
Let us compute the weighted $L^2$ norm of $\varphi_n$. At this aim, we observe that \eqref{scalingw} yields
\[
n^\frac{2}{2-q}\,w_{\Gamma(\beta,R),q}\left(\frac{x}{n}\right)=w_{\Gamma(\beta,n\,R),q}(x),
\]
thus by using a change of variables we get
\[
\begin{split}
\int_{\Gamma(\beta,R)} |\varphi_n|^2\,w_{\Gamma(\beta,R),q}^{q-2}\,dx&= \int_{\Gamma(\beta,R/n)}n^{N-2}\,|\varphi(n\,x)|^2\,w_{\Gamma(\beta,R),q}(x)^{q-2}\,dx\\
&=\int_{\Gamma(\beta,R)} |\varphi|^2\,w_{\Gamma(\beta,n\,R),q}^{q-2}\,dx.
\end{split}
\]
This gives
\[
\lim_{n\to\infty} \frac{\displaystyle \int_{\Gamma(\beta,R)}|\nabla \varphi_n|^2\,dx}{\displaystyle\int_{\Gamma(\beta,R)}|\varphi_n|^2\,w_{\Gamma(\beta,R),q}^{q-2}\,dx}=\lim_{n\to\infty} \frac{\displaystyle \int_{\Gamma(\beta,R)}|\nabla \varphi|^2\,dx}{\displaystyle\int_{\Gamma(\beta,R)}|\varphi|^2\,w_{\Gamma(\beta,n\,R),q}^{q-2}\,dx}.
\]
We now use that 
\[
\lim_{n\to\infty}w_{\Gamma(\beta,n\,R),q}=V,\qquad \mbox{ uniformly on } \Gamma(\beta,R),
\]
thanks to Lemma \ref{lm:coni}. This yields
\[
\lim_{n\to\infty} \frac{\displaystyle \int_{\Gamma(\beta,R)}|\nabla \varphi_n|^2\,dx}{\displaystyle\int_{\Gamma(\beta,R)}|\varphi_n|^2\,w_{\Gamma(\beta,R),q}^{q-2}\,dx}=\frac{\displaystyle \int_{\Gamma(\beta,R)}|\nabla \varphi|^2\,dx}{\displaystyle\int_{\Gamma(\beta,R)} |\varphi|^2\,V^{q-2}\,dx}.
\]
By taking the infimum over $\varphi\in C^\infty_0(\Gamma(\beta,R))$ and using Lemma \ref{lm:1}, we get 
\[
\inf_{\varphi\in C^\infty_0(\Gamma(\beta,R))} \left[\lim_{n\to\infty} \frac{\displaystyle \int_{\Gamma(\beta,R)}|\nabla \varphi_n|^2\,dx}{\displaystyle\int_{\Gamma(\beta,R)}|\varphi_n|^2\,w_{\Gamma(\beta,R),q}^{q-2}\,dx}\right]=\sigma(\beta),
\]
and thus the conclusion of point 1).
\vskip.2cm\noindent
We now show the loss of compactness. Given $t>\mathcal{C}_q(\beta)=\sigma(\beta)$, by point 1) we know that there exists $\varphi\in C^\infty_0(\Gamma(\beta,R))$ such that
\[
\lim_{n\to\infty} \frac{\displaystyle \int_{\Gamma(\beta,R)}|\nabla \varphi_n|^2\,dx}{\displaystyle\int_{\Gamma(\beta,R)}|\varphi_n|^2\,w_{\Gamma(\beta,R),q}^{q-2}\,dx}<t,
\]
where as before we set $\varphi_n(x)=n^{(N-2)/2}\,\varphi(n\,x)$. This shows that the rescaled sequence 
\[
\psi_n=\frac{\varphi_n}{\displaystyle\left(\int_{\Gamma(\beta,R)}|\varphi_n|^2\,w_{\Gamma(\beta,R),q}^{q-2}\,dx\right)^\frac{1}{2}},\qquad n\in\mathbb{N},
\]
belongs to $\mathcal{E}_{\Gamma(\beta,R),q}(t)$ for $n$ large enough. However, this sequence can not converge strongly in the weighted $L^2$ space, since by construction we have (recall \eqref{hardy1} and \eqref{hardy2})
\[
\int_{\Gamma(\beta,R)} w_{\Gamma(\beta,R),q}^{q-2}\,|\psi_n|^2\,dx=1\quad \text{ and }\quad \int_{\Gamma(\beta,R)} |\psi_n|^2\,dx=\frac{1}{n^2}\,\frac{\displaystyle\int_{\Gamma(\beta,R)} |\nabla\varphi|^2\,dx}{\displaystyle\int_{\Gamma(\beta,R)} |\varphi|^2\,w_{\Gamma(\beta,n\,R),q}^{q-2}\,dx}\to 0.
\]
This concludes the proof.
\end{proof}

\end{document}